\documentclass[a4paper,11pt,reqno]{amsart}
\usepackage[utf8]{inputenc}
\usepackage{latexsym,amssymb,amsthm, amsmath}
\usepackage{amssymb,amsfonts,amsmath,mathrsfs,graphicx,multirow}
\usepackage{xcolor}
\usepackage[toc,page]{appendix}
\usepackage{url}
\usepackage{caption}
\newcommand{\Mod}[1]{\ (\mathrm{mod}\ #1)}

\newtheorem{theorem}{Theorem}[section]
\newtheorem{proposition}{Proposition}[section]
\newtheorem{lemma}{Lemma}[section]
\newtheorem{conjecture}{Conjecture}

\newtheorem{corollary}{Corollary}[section]
\theoremstyle{definition}

\begin{document}
\title{Zeros of Dirichlet $L$-functions near the Critical Line}
\author{George Dickinson }
\address{Department of Mathematics, University of Manchester, Manchester M13 9PL, UK}
\email{george.dickinson@manchester.ac.uk}

\begin{abstract}
    We prove an upper bound on the density of zeros very close to the critical line of the family of Dirichlet $L$-functions of modulus $q$ at height $T$. To do this, we derive an asymptotic for the twisted second moment of Dirichlet $L$-functions uniformly in $q$ and $t$. As a second application of the asymptotic formula we prove that, for every integer $q$, at least $38.2\%$ of zeros of the primitive Dirichlet $L$-functions of modulus $q$ lie on the critical line.  
\end{abstract}

\maketitle

\section{Introduction}
The Riemann zeta-function and the Dirichlet $L$-functions are objects of great importance in number theory, and are the subject of many conjectures. One such conjecture is the Density Hypothesis.

\begin{conjecture}[Density Hypothesis]
Let 
\begin{eqnarray*}
N(\sigma, T) \kern-6pt& := & \kern-6pt \{\rho \in \mathbb{C}: \zeta(\rho) = 0, \operatorname{Re}(\rho) \geq \sigma, |\operatorname{Im}(\rho)| \leq T\},\\
N(\sigma, T, \chi) \kern-6pt&:=&\kern-6pt \{\rho \in \mathbb{C}: L(\rho,\chi) = 0, \operatorname{Re}(\rho) \geq \sigma, |\operatorname{Im}(\rho)| \leq T\}.
\end{eqnarray*}
Then,
$$N(\sigma,T) \ll T^{2(1-\sigma)}\log(T)$$
and
$$\sideset{}{^*}\sum_{\chi\Mod{q}}N(\sigma,T;\chi)\ll(qT)^{2(1-\sigma)}\log(qT)$$
for $\sigma \in [1/2, 1]$, $q\geq 2$ and $T \geq 3$. Here the implied constants are absolute.
\end{conjecture}
If the Density Hypothesis is true, then it could be used as a replacement for the Riemann Hypothesis for various applications to the distribution of primes. For example, see Section $10.5$ of \cite{IandK}. In this paper, we consider the Density Hypothesis for Dirichlet $L$-functions.

Notable progress has been made towards proving this conjecture in various ranges of $\sigma$. While some techniques are more appropriate for proving density results closer to $\sigma =1 $ (see \cite{Pintz} for recent results and further references), this paper is concerned with a range of $\sigma$ very close to $1/2$. To date, the best density theorem in this context is by Montgomery (Theorem 12.1 in \cite{Topics}) which states that 
$$\sum_{\chi\Mod{q}}N(\sigma,T, \chi) \ll(qT)^{\frac{3(1-\sigma)}{2-\sigma}}\log^9(qT)$$
for $\sigma \in [1/2, 4/5]$, $q \geq 1$ and $T \geq 2$.

Montgomery used zero detecting polynomials, while we shall be using moments of $L$-functions to prove the following theorem.

\begin{theorem}
\label{theoremdensity}
For all $\kappa <69/128$ and $\epsilon > 0$,
$$\sideset{}{^*}\sum_{\chi\Mod{q}}N(\sigma,T;\chi)\ll_\epsilon(qT)^{2-2\sigma}\log^5(qT) + (qT)^{1+\kappa(1-2\sigma)}\log(qT)^2\log\log(qT) $$
for all $q \geq 2$, $T \geq q^\epsilon$, and $\sigma \in [1/2, 1]$. 

\end{theorem}
This improves Montgomery's result in the range 
$$\frac{1}{2}\leq \sigma \leq \frac{1}{2} + \frac{27\log\log(qT)}{\log(qT)}.$$

This result is proved using an asymptotic for the second moment of the Dirichlet $L$-functions in both the $q$ and the $t$-aspect, twisted by a mollifier. As the Dirichlet $L$-functions have different functional equations depending on whether their associated Dirichlet character is odd or even (i.e. odd characters satisfy $\chi(-1)=-1$, while even satisfy $\chi(-1)=1$), we split the sum over the characters into sums over the odds and the evens. The sum over the odd characters and the sum over the even characters are denoted as 
$$\sideset{}{^-}\sum_{\chi \Mod{q}} \ \text{and} \ \sideset{}{^+}\sum_{\chi \Mod{q}}$$
respectively. For the sake of simplicity, we focus on just the even sum then address the minor differences in proof needed for the odd sum in Section \ref{odd}. In total, there are $\phi^*(q)$ principal characters of modulus $q$. To distinguish the principal character of modulus $q$, we write it as $\chi_{0,q}$.
\begin{theorem}
\label{moment}
Let $q$ be a positive integer with $T \gg q^\epsilon$. Let $\psi(t)$ be a smooth real valued function supported on $[1,2]$ with $\psi^{(j)}(t) \ll T^\epsilon$. Let $\alpha, \beta \in \mathbb{C}$ satisfy $\alpha, \beta \ll \log\log(qT)/\log(qT)$. Suppose that $1/2<\kappa < 1/2 + 1/66$. For all $n \in \mathbb{N}, \alpha_n, \beta_n \in \mathbb{C}$ such that $\alpha_n, \beta_n \ll n^\epsilon$,
\begin{multline*}
\frac{1}{\phi^*(q)T}\int\sideset{}{^+}\sum_{\chi\Mod{q}}L(1/2 + \alpha + it, \chi)L(1/2 + \beta - it, \bar{\chi})\sum_{a,b \leq (qT)^\kappa}\frac{\alpha_a\beta_b\chi(a)\bar{\chi}(b)}{a^{1/2 + it}b^{1/2 - it}}\psi\left(\frac{t}{T}\right)dt\\
= \frac{\hat{\psi}(0)}{2}L(1 + \alpha + \beta, \chi_{0,q})\sum_{\substack{ad,bd \leq (qT)^\kappa\\ (a,b)=1\\ (abd,q) = 1}} \frac{\alpha_{ad}\beta_{bd}}{a^{1+\beta}b^{1+\alpha}d} + \frac{1}{2T}\left(\frac{q}{\pi}\right)^{-\alpha-\beta}L(1-\alpha - \beta, \chi_{0,q})\\
\times\sum_{\substack{ad,bd \leq (qT)^\kappa\\ (a,b)=1\\(abd,q)=1}} \frac{\alpha_{ad}\beta_{bd}}{a^{1-\alpha}b^{1-\beta}d}\int\frac{\Gamma\left(\frac{1/2-\alpha-it}{2}\right)\Gamma\left(\frac{1/2-\beta+it}{2}\right)}{\Gamma\left(\frac{1/2+\alpha+it}{2}\right)\Gamma\left(\frac{1/2+\beta-it}{2}\right)}\psi\left(\frac{t}{T}\right)dt +O_\epsilon\left((qT)^{-\epsilon}\right).
\end{multline*}
 
\end{theorem}
By introducing the small shifts $\alpha$ and $\beta$, we not only derive a more general result, but calculating the second moment (by letting the shifts tend to zero and taking the limit) is actually easier. In the case that $\alpha = -\beta$, then the above result should be considered as a limit.

A natural choice of mollifier (and one that we shall use to prove Theorem \ref{theoremdensity} is $M(s,\chi) = \sum_{n}\frac{\mu(n)f(n)\chi(n)}{n^s}$ where $f(n)$ is some smoothing function. In this case, it is possible to exploit the properties of the M\"obius function to get a smaller error term.
\begin{theorem}
\label{moment2}
Suppose that the conditions of Theorem \ref{moment} hold, with the added assumption that $\alpha(n) = \mu(n)f(n)$ for some smooth bounded function $f(x)$ with $f^{'}(x)\ll_\epsilon x^{-1+\epsilon}$. Then the same result holds for $1/2 <\kappa < 1/2 + 5/128$.

\end{theorem}

It is worth noting that as $T$ grows arbitrarily large compared to $q$ then by using similar techniques as in \cite{Conrey1989} $\kappa$ can be increased up to a limit of $\kappa < 4/7 \approx 0.571$.

Asymptotics for twisted second moments that break the half barrier are not new, as Bettin, Chandee and Radziwi\l{}\l{} achieved this in the $t$-aspect for the Riemann zeta-function in \cite{BCR}  and Bui, Pratt, Robles and Zaharescu in the $q$-aspect in \cite{Bui2020}. However, finding an asymptotic that is uniform in both has its own challenges, mostly due to terms that are negligible in the $q$-aspect no longer being negligible when the $t$-aspect is introduced. Previous results in just the $q$ aspect only work when $q$ is prime, while this result applies to all positive integers $q$.

We demonstrate a second application of Theorem \ref{moment2}, using it to prove a result on the proportion of simple zeros on the critical line. Let $N(T,\chi)$ denote the number of zeros $\rho = \beta + i\gamma$ of the Dirichlet $L$-function $L(s, \chi)$ for a character $\chi$ of conductor $q$, with $0 < \gamma < T$. Let $N_0(T, \chi)$ denote the number of these zeros that are simple with $\beta = 1/2$.
\begin{theorem}
\label{Levinson}
Define 
$$N(T,q) = \frac{1}{\phi^*(q)}\sideset{}{^*}\sum_{\chi\Mod{q}}N(T,\chi), \ \text{ and } \ N_0(T,q) = \frac{1}{\phi^*(q)}\sideset{}{^*}\sum_{\chi\Mod{q}}N_0(T,\chi).$$
Then for $\kappa < 1/2 + 5/128$ we have
$$\frac{N_0(T,q)}{N(T,q)} \geq 1-\frac{1}{R}\log(c(P,Q,R)) + o(1)$$
where 
$$c(P,Q,R) = 1+ \frac{1}{\kappa}\int_0^1\int_0^1 e^{2Rv}\left(\frac{d}{dx}e^{R\kappa x}P(x+u)Q(v+\kappa x)|_{x=0}\right)^2dudv,$$
$R>0$ is a positive constant, $P(x)$ is polynomial with $P(0)=0, P(1)=1$, and $Q(x)$ is a real linear polynomial with $Q(0)=1$.

\end{theorem}
By choosing $Q$ to be a non-linear polynomial, we would obtain a lower bound on the number of zeros on the critical line, simple or otherwise. In fact it is conjectured that all non-trivial zeros are simple. By choosing $R,P$, and $Q$ optimally, we arrive at the following corollary.
\begin{corollary}
$$\liminf_{qT\rightarrow \infty}\frac{N_0(T,q)}{N(T,q)} \geq 0.382.$$
\end{corollary}
Informally, this means that for integer $q$ at least 38.2\% of zeros up to a large height $T$ of the primitive Dirichlet $L$-functions of modulus $q$ lie on the critical line as we vary $q$ such that $\log(q) \ll \log(T)$.

Theorem \ref{Levinson} comes from applying Levinson's method to Theorem \ref{moment2}. Levinson's method is an elegant and widely used technique for determining the proportion of critical zeros of an $L$-function.  See \cite{CIS} for a nice demonstration of the method, and \cite{LevinsonY} for an elegant application of the method to the Riemann zeta-function.

Levinson's method has been used by Conrey in just the $t$-aspect in \cite{Conrey1989} to show that at least 40.7\% non-trivial zeros of the Riemann zeta-function are critical (this has since been improved to 41.7\% in \cite{fivetwelfths}), while in \cite{CIS} Conrey, Iwaniec, and Soundararajan consider the $q$-aspect, averaged over $q \leq Q$ to conclude that at least 56\% of low-lying zeros lie on the critical line (see also  \cite{sono}). In comparison, our result is uniform in $q$ and $t$, and does not require averaging over $q\leq Q$.

We begin by proving Theorem \ref{moment} and Theorem \ref{moment2} in Section \ref{moment proof}. Then we focus on the applications and prove Theorem \ref{theoremdensity} in Section \ref{Density proof} and Theorem \ref{Levinson} in Section \ref{Levinson proof}.

Throughout this paper we shall use the convention that $\epsilon$ is an arbitrarily small positive constant that may change value between lines.

\section{The Twisted Second Moment}
\label{moment proof}
\subsection{Initial Manipulations}
\begin{lemma}[Approximate Functional Equation]
Let $\chi$ be an even primitive character. Then we have the approximate functional equation
$$L\left( \frac{1}{2} + \alpha + it, \chi\right)L\left( \frac{1}{2} + \beta - it, \bar{\chi}\right) = \sum_{m,n \geq 1}\frac{\chi(m)\bar{\chi}(n)}{m^{1/2 + \alpha + it}n^{1/2 + \beta - it}}V_+\left(\frac{\pi mn}{q}, t\right)$$
$$+ \left(\frac{q}{\pi}\right)^{-\alpha - \beta}\sum_{m,n\geq 1}\frac{\chi(m)\bar{\chi}(n)}{m^{1/2-\beta+it}n^{1/2-\alpha-it}}V_-\left(\frac{\pi mn}{q}, t\right)$$
where 
$$V_\pm(x,t) = \frac{1}{2\pi i}\int_{(\epsilon)}X_{\pm}(s,t)x^{-s}\frac{ds}{s}$$
and 
$$X_\pm(s,t) = G(s)\frac{\Gamma(\frac{1/2 \pm(\alpha + it) + s}{2})\Gamma(\frac{1/2 \pm (\beta - it) + s}{2})}{\Gamma(\frac{1/2 + \alpha + it }{2})\Gamma(\frac{1/2 + \beta - it }{2})}$$
and $G(s)$ is a function that is even, entire, of rapid decay in any fixed strip and with $G(0) = 1, \ G(\pm(\alpha + \beta)/2)=0$. 
\end{lemma}
The proof is standard. For example, see Theorem 5.3 of \cite{IandK}. 
\begin{lemma}
\label{gamma}
For all $i,j,C \geq 0$, 
\begin{equation}
\label{eqn: V+}
x^jt^i\frac{\partial^{i+j}}{\partial x^j\partial t^i}V_{+}(x,t)\ll _{i,j,C}(1+|x/t|)^{-C}
\end{equation}
\begin{equation}
\label{V-}
x^jt^i\frac{\partial^{i+j}}{\partial x^j\partial t^i}V_{-}(x,t)\ll _{i,j,C}|t|^{-\operatorname{Re}(\alpha + \beta)}(1+|x/t|)^{-C}.
\end{equation}
\end{lemma}
The proof is a simple application of Stirling's approximation applied to
$$\frac{\Gamma(\frac{1/2 +\alpha + it + s}{2})\Gamma(\frac{1/2 + \beta - it + s}{2})}{\Gamma(\frac{1/2 + \alpha + it }{2})\Gamma(\frac{1/2 +\beta - it }{2})} =  t^{s}(1 + O((1+|t|)^{-1})).$$

\begin{lemma}[Orthogonality]
\label{orthog}
Suppose that $(m,q) = 1$, then $$\sideset{}{^+}\sum_{\chi \Mod{q}}\chi(m) = \frac{1}{2}\sum_{\substack{uw=q\\ m \equiv \pm 1 \Mod{w}}}\mu(u)\phi(w).$$
\end{lemma}
The proof of this result is standard. See for example, (3.1) and (3.2) of \cite{sarnak}.
 
Applying Lemma \ref{orthog} to the approximate functional equation gives 
\begin{multline*}
\sideset{}{^+}\sum_{\chi\Mod{q}}L(1/2 + \alpha + it, \chi)L(1/2 + \beta - it, \bar{\chi})\sum_{a,b \leq (qT)^\kappa}\frac{\alpha_a\beta_b\chi(a)\bar{\chi}(b)}{a^{1/2 + it}b^{1/2 - it}} =\\
\frac{1}{2}\sum_{w|q}\mu\left(\frac{q}{w}\right)\phi(w)\left(
\int\sum_{\substack{a,b,m,n\\am \equiv \pm bn \Mod{w}\\ (abmn,q) = 1}}\frac{\alpha_a\beta_b}{(ab)^{1/2}m^{1/2 +\alpha}n^{1/2+\beta}}\left(\frac{bn}{am}\right)^{it}V_+\left(\frac{\pi mn}{q},t\right)\psi\left(\frac{t}{T}\right)dt\right. \\
+ \left.\left(\frac{q}{\pi}\right)^{-\alpha-\beta}\int\sum_{\substack{a,b,m,n\\am \equiv \pm bn \Mod{w}\\ (abmn,q) = 1}}\frac{\alpha_a\beta_b}{(ab)^{1/2}m^{1/2 -\beta}n^{1/2-\alpha}}\left(\frac{bn}{am}\right)^{it}V_-\left(\frac{\pi mn}{q},t\right)\psi\left(\frac{t}{T}\right)dt \right).
\end{multline*}
Hence
$$\int\sideset{}{^*}\sum_{\chi\Mod{q}}L(1/2 + \alpha + it, \chi)L(1/2 + \beta - it, \bar{\chi})\sum_{a,b \leq (qT)^\kappa}\frac{\alpha_a\beta_b\chi(a)\bar{\chi}(b)}{a^{1/2 + it}b^{1/2 - it}}\psi\left(\frac{t}{T}\right)dt$$
$$= D^+ + O^++\left(\frac{q}{\pi}\right)^{-\alpha - \beta}(D^- + O^-)$$

where 
$$D^+:=\frac{\phi^*(q)}{2}\int\sum_{\substack{a,b \leq (qT)^{\kappa}\\ am = bn\\(abmn,q)=1}}\frac{\alpha_a\beta_b}{(ab)^{1/2}m^{1/2 + \alpha}n^{1/2 + \beta}}V_+\left(\frac{\pi mn}{q},t\right)\psi\left(\frac{t}{T}\right)dt$$
and 
$$O^+ := \frac{1}{2}\sum_{w|q}\mu\left(\frac{q}{w}\right)\phi(w)\int\sum_{\substack{a,b,m,n\\am \equiv \pm bn \Mod{w}\\am \neq bn\\ (abmn,q) = 1}}\frac{\alpha_a\beta_b}{(ab)^{1/2}m^{1/2 +\alpha}n^{1/2+\beta}}\left(\frac{bn}{am}\right)^{it}V_+\left(\frac{\pi mn}{q},t\right)\psi\left(\frac{t}{T}\right)dt .$$

$D^-$ and $O^-$ are obtained by from $D^+$ and $O^+$ by substituting $\alpha, \beta \rightarrow -\beta, -\alpha$ and replacing $V_+$ by $V_-$. As the $D^+$ and $O^+$ cases are almost identical to the $D^-$ and $O^-$ cases, we shall only demonstrate the former.

\subsubsection{The Diagonals}
As the diagonals are made up of sums over the condition $am = bn$, we may write $a,b, m,n = ad,bd, bn', an'$ with $(a,b) = 1$ and $(n',q)=1$. Hence, by relabelling, $D^+$ is
\begin{eqnarray*}
\kern-6pt &= & \kern-6pt  \frac{\phi^*(q)}{2}\sum_{\substack{ad,bd \leq (qT)^\kappa\\ (a,b)=1\\(abd,q)=1}}\sum_{\substack{n\geq 1\\(q,n)=1}}\frac{\alpha_{ad}\beta_{bd}}{a^{1+\beta}b^{1+\alpha}dn^{1+ \alpha + \beta}}\int V_+\left(\frac{\pi abn^2}{q},t\right)\psi\left(\frac{t}{T}\right)dt\\
\kern-6pt & = & \kern-6pt   \frac{\phi^*(q)}{2}\sum_{\substack{ad,bd \leq (qT)^\kappa\\ (a,b)=1\\(abd,q)=1}}\sum_{\substack{n\geq 1\\(q,n)=1}} \frac{\alpha_{ad}\beta_{bd}}{a^{1+\beta}b^{1+\alpha}dn^{1+ \alpha + \beta}}\frac{1}{2\pi i}\int\int_{(2)}X_+(s,t)\left(\frac{q}{abn^2\pi}\right)^s\psi\left(\frac{t}{T}\right)\frac{ds}{s}dt\\
\kern-6pt & = & \kern-6pt   \frac{\phi^*(q)}{2}\sum_{\substack{ad,bd \leq (qT)^\kappa\\ (a,b)=1\\(abd,q)=1}}\sum_{\substack{n\geq 1\\(q,n)=1}} \frac{\alpha_{ad}\beta_{bd}}{a^{1+\beta}b^{1+\alpha}d}\frac{1}{2\pi i}\int\int_{(2)}X_+(s,t)\left(\frac{q}{ab\pi}\right)^sn^{-(1+\alpha + \beta + 2s)}\psi\left(\frac{t}{T}\right)\frac{ds}{s}dt\\
\kern-6pt & = & \kern-6pt   \frac{\phi^*(q)}{2}\sum_{\substack{ad,bd \leq (qT)^\kappa\\ (a,b)=1\\(abd,q)=1}} \frac{\alpha_{ad}\beta_{bd}}{a^{1+\beta}b^{1+\alpha}d}\frac{1}{2\pi i}\int\int_{(\epsilon)}X_+(s,t)\left(\frac{q}{ab\pi}\right)^sL(1+\alpha + \beta + 2s,\chi_{0,q})\psi\left(\frac{t}{T}\right)\frac{ds}{s}dt.
\end{eqnarray*}

Note that we chose the contour of integration in the $V$-function to be Re($s$) $= 2$ at first so that the sum over $n$ converges to the $L$-function, and then moved the contour back to Re($s$)$= \epsilon$ with the pole at $s= -(\alpha + \beta)/2$ being cancelled by the zero coming from $X_+(s,t)$. Similarly 
$$D^- = \frac{\phi^*(q)}{2}\sum_{\substack{ad,bd \leq (qT)^\kappa\\ (a,b)=1\\(abd,q)=1}} \frac{\alpha_{ad}\beta_{bd}}{a^{1-\alpha}b^{1-\beta}d}\frac{1}{2\pi i}\int\int_{(\epsilon)}X_-(s,t)\left(\frac{q}{ab\pi}\right)^sL(1-\alpha-\beta+ 2s,\chi_{0,q})\psi\left(\frac{t}{T}\right)\frac{ds}{s}dt.$$

\subsubsection{Off-diagonals}

The remaining terms (i.e. when $am \neq bn$) are the off-diagonals. The following lemma will allow us to show that the terms in the sum with $am$ and $bn$ sufficiently far away from each other will contribute a negligible amount. 

\begin{lemma}
\label{tint}
Suppose that $\psi:[1,2] \rightarrow \mathbb{R}$ is a smooth function with derivative $\frac{d^K}{dt^K}\psi(t) \ll_{K,\epsilon}  T^{\epsilon}$. Then,
$$\int V_{\pm}\left(\frac{\pi mn}{q},t\right)\left(\frac{bn}{am}\right)^{it}\psi\left(\frac{t}{T}\right)dt\ll_{K} |\log(bn/am)|^{-K}T^{1-K+\epsilon}$$
and hence this integral is vanishingly small unless 
$$1- T^{\epsilon - 1} < \frac{bn}{am} < 1+ T^{\epsilon - 1}.$$
\end{lemma}
\begin{proof}

By (\ref{eqn: V+}), for $t \in [T,2T]$
$$\frac{\partial^K}{\partial t^K}V_\pm(x,t)\psi\left(\frac{t}{T}\right)dt\ll_{K,C}  T^{-K+\epsilon}\left(1+\frac{|x|}{T}\right)^{-C}$$
for all $C>0$. This implies 
$$\int V_{\pm}\left(\frac{\pi mn}{q},t\right)\left(\frac{bn}{am}\right)^{it}\psi\left(\frac{t}{T}\right)dt\ll_{K} |\log(bn/am)|^{-K}T^{1-K+\epsilon}.$$

By taking $K\rightarrow \infty$, this becomes negligibly small unless $|\log(bn/am)| \ll T^{-1}$. Taking the Taylor expansion of $\log(1+x) = x + O(x^2)$ for $|x| < 1$ to see that the $t$-integral is vanishingly small unless 
$$1- T^{\epsilon - 1} < \frac{bn}{am} < 1+ T^{\epsilon - 1}.$$
\end{proof}

To help restrict to these non-negligible cases, we introduce a dyadic partition of unity to the sums over $m$ and $n$: let $W$ be a smooth non-negative function supported in [1,2] such that 
$$\sum_{M}W\left(\frac{m}{M}\right)=1,$$
where $M$ runs over a sequence of real numbers with $|\{M: X^{-1}\leq M \leq X\}| \ll  \log X$. By the rapid decay of $V_\pm$, in (\ref{eqn: V+}) and (\ref{V-}) we may assume that $MN \ll  (qT)^{1+\epsilon}$. We also split up the mollifying coefficients $\alpha_n, \beta_n$ dyadically, supposing that $\alpha_n(A)$ is supported on $n \in [A,2A]$ and $\beta_n(B)$ is supported on $n\in [B,2B]$ i.e. $\alpha_a = \sum_{A}\alpha_a(A)$ and by the assumptions in Theorem \ref{moment}, $A, B \ll (qT)^\kappa$. In the next section, we extract the main term from the off-diagonal terms, and bound the rest into an error term.

\subsection{Main Propositions}
When the mollifier is short enough, a trivial bound is sufficient to bound the contribution from the off-diagonal term. However, to break the half-barrier, a more sophisticated method is needed as the off-diagonals begin to contribute to the main term. The trivial bound shall be of use later on in the proof.

\begin{lemma}[Trivial bound]
\label{triv}
For pairwise co-prime $f,h,w|q$ and $\forall a,b \ \alpha_a,\beta_b  \ll (qT)^\epsilon $
$$\sum_{\substack{amf = \pm bnh + wr\\ a,b,m,n,r \asymp A,B,\frac{M}{f},\frac{N}{h},\frac{R}{w}\\(ab,q)=1}}\alpha_a\beta_b \ll \frac{AR}{w}\left(1 + \frac{M}{fh}\right)(qT(AM + R))^\epsilon$$
\end{lemma}
\begin{proof}

We bound this sum trivially by summing over $a$ and $r$. Then we sum over $m\equiv\overline{af}wr \Mod{h}$, of which there are $\ll 1 + M/fh$ possible values of $m$. Then we bound the sums over $b$ and $n$ using the divisor bound. 
\end{proof}

The next proposition shows how the off-diagonals contribute to the main term.

\begin{proposition}
Let $T \gg q^\epsilon$ for a positive integer $q$, and $w|q$. Let $\psi(t)$ amd $W(x)$ be smooth real valued functions supported on $[1,2]$ such that for all $j \geq 0$ $\psi^{(j)}(t) \ll T^\epsilon$ and $W^{(j)}(x) \ll_\epsilon (qT)^\epsilon$. Let $\alpha, \beta \in \mathbb{C}$ satisfy $\alpha, \beta \ll \log\log(qT)/\log(qT)$. Suppose that $1/2< \kappa < 1/2 + 1/66$. Suppose that for positive constants $1\leq A,B \leq (qT)^\kappa$, $1\leq MN \ll  (qT)^{1+\epsilon}$ with $AM \asymp BN$, we have complex sequences $ \alpha_a(A), \beta_b(B) \in \mathbb{C}$ with support on $[A, 2A]$ and $[B,2B]$ respectively such that $\alpha_n, \beta_n \ll n^\epsilon$.\\
Define $S^+_{w}(A,B,M,N)$ as
$$\int\sum_{\substack{a,b,m,n\\am \equiv \pm bn \Mod{w}\\am \neq bn\\ (abmn,q) = 1}}\frac{\alpha_a(A)\beta_b(B)}{(ab)^{1/2}m^{1/2 +\alpha}n^{1/2+\beta}}\left(\frac{bn}{am}\right)^{it}W\left(\frac{m}{M}\right)W\left(\frac{n}{N}\right)V_+\left(\frac{\pi mn}{q},t\right)\psi\left(\frac{t}{T}\right)dt$$
then $\frac{1}{\phi^*(q)T}\sum_{w|q}\mu\left(\frac{q}{w}\right)\phi(w)S_w^+(A,B,M,N)$ is equal to
$$\frac{1}{\phi^*(q)T}\sum_{w|q}\mu\left(\frac{q}{w}\right)\phi(w)\sum_{\substack{f|q\\(f,w)}}\mu(f) \frac{\phi(q/f)}{q}\left(\mathcal{M}^+_{w,f}(A,B,M,N) + \mathcal{M}^-_{w,f}(A,B,M,N)\right) + O_\epsilon\left((qT)^{-\epsilon}\right)$$

where 
$$\mathcal{M}^+_{w,f}(A,B,M,N) = \int\sum_{\substack{d \geq 1 \\(d,q)=1}}\sum_{\substack{r\geq 1\\(r,f)=1}}\sum_{\substack{a,b\\(a,b)=1\\ (ab,q) = 1}}\frac{\alpha_{d a}(A)\beta_{db}(B)}{(ab)^{1/2}d} \left(\frac{b}{a}\right)^{it}\frac{1}{2\pi i}\int_{(\epsilon)}\int X_+(s,t)$$
$$\left(\frac{q}{\pi}\right)^s \left(\frac{abx - wr}{b}\right)^{-(1/2+s+\beta -it)}(bx)^{-(1/2 + \alpha+s+it)}W\left(\frac{bx}{M}\right)W\left(\frac{abx-wr}{bN}\right)\psi\left(\frac{t}{T}\right)dx\frac{ds}{s}dt$$ 
and 
$$\mathcal{M}^-_{w,f}(A,B,M,N) = \int\sum_{\substack{d \geq 1 \\(d,q)=1}}\sum_{\substack{r\geq 1\\(r,f)=1}}\sum_{\substack{a,b\\(a,b)=1\\ (ab,q) = 1}}\frac{\alpha_{d a}(A)\beta_{db}(B)}{(ab)^{1/2}d} \left(\frac{b}{a}\right)^{it}\frac{1}{2\pi i}\int_{(\epsilon)}\int X_+(s,t)$$
$$\left(\frac{q}{\pi}\right)^s \left(\frac{wr-abx}{b}\right)^{-(1/2+s+\beta -it)}(bx)^{-(1/2 + \alpha+s+it)}W\left(\frac{bx}{M}\right)W\left(\frac{wr-abx}{bN}\right)\psi\left(\frac{t}{T}\right)dx\frac{ds}{s}dt$$

\end{proposition}

\begin{proof}
We begin by writing the $am \equiv \pm bn \Mod{w}$ condition as $am = \pm bn + wr$. As $am \neq bn$, $r$ must be non-zero, and by Lemma (\ref{tint}) we may assume that $|r| \leq 2AMw^{-1}T^{\epsilon - 1}$, so we sum over $0 < |r| \leq R/w$ where $R := 2AMT^{-1 + \epsilon}$. We remove the $(mn,q)=1$ condition as follows: for any smooth function $F(a,b,m,n)$ for a fixed $a,b,w$ and $r$,
$$\sum_{\substack{am=\pm bn+wr\\(abmn,q)=1}}F(a,b,m,n)= \sum_{\substack{f|q\\(f,wr)=1}}\mu(f)\sum_{\substack{am= \pm bn+wr\\(abm,q)=1\\f|n}}F(a,b,m,n)$$
$$=\sum_{\substack{f|q\\(f,rw)}}\mu(f)\sum_{\substack{am\equiv wr \Mod{bf}\\(abm,q)=1}}F\left(a,b,m,\frac{\mp(am-wr)}{b}\right)$$

Note that if $(f,rw)>1$ then the sum is empty as then $(am,q)>1$. Given this, we can then relax the condition that $(m,q)=1$ to $(m,q/f)=1$, as $m$ must be coprime to $f$ by the residue condition $am \equiv wr \Mod{bf}$. Suppose for contradiction that $p|(f,q/f)$ then $p^2|q$ and hence as $q/w$ is square free it must be the case that $p|w$. Hence $p$ can not divide $f$, so $(f,q/f)=1$. So 
$$\sum_{\substack{f|q\\(f,rw)}}\mu(f)\sum_{\substack{am\equiv wr \Mod{bf}\\(abm,q)=1\\ (a,b)=1}} = \sum_{\substack{f|q\\(f,rw)}}\mu(f)\sum_{\substack{am\equiv wr \Mod{bf}\\(ab,q)=1\\(m,q/f)=1\\ (a,b)=1}} = \sum_{\substack{f|q\\(f,rw)}}\mu(f)\sum_{u \in \left(\frac{\mathbb{Z}}{(q/f)\mathbb{Z}}\right)^*}\sum_{\substack{m\equiv \bar{a}wr \Mod{bf}\\ m \equiv \bar{a}u \Mod{q/f}\\(ab,q)=1\\ (a,b)=1}}.$$
Let $x_1 \equiv \overline{aq/f} \Mod{bf}$ and $x_2 \equiv \overline{abf} \Mod{q/f}$, so that by appealing to the Chinese remainder theorem
$m\equiv \bar{a}wr \Mod{bf}$ and $ m \equiv \bar{a}u \Mod{q/f}$ iff $m \equiv wrx_1q/f + ux_2bf \Mod{bq}$. Then we apply Poisson summation to find that 
$$\sum_{\substack{am=\pm bn+wr\\(abmn,q)=1}}F(a,b,m,n) =  \sum_{\substack{f|q\\(f,rw)}}\mu(f)\sum_{u \in \left(\frac{\mathbb{Z}}{(q/f)\mathbb{Z}}\right)^*}\sum_{m \equiv wrx_1q/f + ux_2bf \Mod{bq}}F\left(a,b,m,\frac{\mp(am-wr)}{b}\right)$$
\begin{eqnarray*}
\kern-6pt&=&\kern-6pt  \sum_{\substack{f|q\\(f,rw)}}\mu(f)\sum_{u \in \left(\frac{\mathbb{Z}}{(q/f)\mathbb{Z}}\right)^*} \sum_{g\in \mathbb{Z}}\frac{e\left(\frac{-gwrx_1q/f - gux_2bf}{bq}\right)}{q}\int F\left(a,b,bx,\mp\frac{(abx-wr)}{b}\right)e(gx/q)dx\\
\kern-6pt&=&\kern-6pt \sum_{\substack{f|q\\(f,rw)}}\mu(f)\sum_{u \in \left(\frac{\mathbb{Z}}{(q/f)\mathbb{Z}}\right)^*} \sum_{g\in \mathbb{Z}}\frac{e\left(-\frac{gwr\overline{aq/f}}{bf}- \frac{gu\overline{abf}}{q/f}\right)}{q}\int F\left(a,b,bx,\frac{\mp(abx-wr)}{b}\right)e(gx/q)dx
\end{eqnarray*}
Summing over $u$ gives a Ramanujan sum i.e.
$$\sum_{\substack{f|q\\(f,rw)}}\mu(f) \sum_{g\in \mathbb{Z}}e\left(\frac{-gwr\overline{aq/f}}{bf}\right)c_{q/f}(g)\frac{1}{q}\int F\left(a,b,bx,\frac{\mp(abx-wr)}{b}\right)e(gx/q)dx$$
where $c_q(n) = \sum_{k|(n,q)}\mu(q/k)k$.
The main term comes from when $g=0$ i.e.
$$\sum_{\substack{f|q\\(f,rw)}}\mu(f) \frac{\phi(q/f)}{q}\int F\left(a,b,bx,\frac{\mp(abx-wr)}{b}\right)dx$$
For the contributions when $g\neq 0$, we expand out $c_{q/f}(g)$ to get
$$\sum_{\substack{f|q\\(f,rw)}}\mu(f) \sum_{k|(q/f)}\mu\left(\frac{q}{kf}\right)\sum_{g \neq 0}e\left(\frac{-gkwr\overline{aq/f}}{bf}\right)\frac{k}{q}\int F\left(a,b,bx,\frac{\mp(abx-wr)}{b}\right)e(gkx/q)dx.$$
As $k| (q/f)$, write $q=kfh$ to get
$$\sum_{\substack{f|q\\(f,rw)}}\mu(f) \sum_{h|(q/f)}\mu\left(h\right)\sum_{g \neq 0}e\left(\frac{-gwr\overline{ah}}{bf}\right)\frac{1}{hf}\int F\left(a,b,bx,\frac{\mp(abx-wr)}{b}\right)e(gx/hf)dx.$$
We will be able to bound the size of $g$ by integrating by parts $j$ times i.e.
$$\int(bx)^{-(1/2 + \alpha + s + it)}\left(\frac{\mp(abx-wr)}{b}\right)^{-1/2-\beta+it-s}W\left(\frac{bx}{M}\right)W\left(\frac{\mp(abx-wr)}{bN}\right)e\left(\frac{gx}{fh}\right)dx$$
is
$$O_{\epsilon,j}\left( (qT)^\epsilon \frac{dM^{1/2}}{BN^{1/2}}\left|\frac{Bfh}{gdM}\right|^j\right)$$
for any fixed $j \geq 0$. So we may restrict the sum to $0 <|g|\leq Gfh/d$ where $G = \frac{B}{M}T^\epsilon$. Hence
$$\sum_{w|q}\mu\left(\frac{q}{w}\right)\phi(w)S_w^+(A,B,M,N)$$
is equal to 
$$ \sum_{w|q}\mu\left(\frac{q}{w}\right)\phi(w)\sum_{\substack{f|q\\(f,rw)}}\mu(f) \frac{\phi(q/f)}{q}\left(\mathcal{M}^+_{w,f}(A,B,M,N) + \mathcal{M}^-_{w,f}(A,B,M,N)\right) + \mathcal{E}$$
where 
\begin{equation}
\label{err}
\begin{split}
\mathcal{E} = \sum_{w|q}\mu\left(\frac{q}{w}\right)\phi(w)\sum_{\substack{f|q\\(f,w)}}\mu(f)\sum_{h|(q/f)}\mu(h)\mathcal{E}_{w,f,h}(A,B,M,N)
\end{split}
\end{equation}
with 
\begin{multline}
\label{E}
\mathcal{E}_{w,f,h}(A,B,M,N):=\sum_{\substack{d \geq 1\\(d,q)=1}}\sum_{\substack{0<|r|\leq R/wd\\(r,f)=1}}\sum_{0< |g|\leq Gfh/d}\sum_{\substack{(a,b)=1\\(ab,q)=1}}\frac{\alpha_{d a}(A)\beta_{db}(B)}{(ab)^{1/2}d} \left(\frac{b}{a}\right)^{it}\frac{e\left(\frac{-gwr\overline{ah}}{bf}\right)}{2\pi ihf}\\
\int_{(\epsilon)}\int X_+(s,t) \left(\frac{q}{\pi}\right)^s \left(\frac{wr-abx}{b}\right)^{-(1/2+s+\beta -it)}(bx)^{-(1/2 + \alpha+s+it)}W\left(\frac{bx}{M}\right)W\left(\frac{wr-abx}{bN}\right)\\
\psi\left(\frac{t}{T}\right)e\left(\frac{gx}{hf}\right)dx\frac{ds}{s}dt.
\end{multline}

For each $w,h$ and $f$, we treat the error term differently depending on the size of $hf$. In short, when $hf$ is large compared to $qT$, then the contribution to the error term (and the main term) can be trivially bounded to be small enough to be absorbed into the error term. When $hf$ is small, we need a more sophisticated method which is an adaptation of Bettin and Chandee's Theorem 1 in \cite{BetChan}. To this aim, define 
$$\gamma := \frac{\log(hf)}{\log(q)}.$$

The contribution to the main term and the error term for a fixed $w, f$ and $h$ can be bounded trivially by reversing the Poisson summation to get the contribution
\begin{multline*}
\frac{\phi(w)}{\phi^*(q)T}\sum_{\substack{d \geq 1\\(d,q)=1}}\sum_{0<|r|\leq R/wd}\sum_{\substack{a,b\\(a,b)=1\\(ab,q)=1}}\int\sum_{\substack{m,n\\amh \equiv \pm bnf \Mod{w}\\amh \neq bnf}}\frac{\alpha_a(A)\beta_b(B)}{(ab)^{1/2}(mh)^{1/2 +\alpha}(nf)^{1/2+\beta}}\left(\frac{bnf}{amh}\right)^{it}\\
W\left(\frac{mh}{M}\right)W\left(\frac{nf}{N}\right)V_+\left(\frac{\pi mnhf}{q},t\right)\psi\left(\frac{t}{T}\right)dt.
\end{multline*}
Using the trivial bound Lemma \ref{triv} we see that 
\begin{eqnarray*}
    \frac{\mathcal{E}_{w,f,h}(A,B,M,N)\phi(w)}{\phi^*(q)T}\kern-6pt&\ll_\epsilon&\kern-6pt \frac{1}{(ABMN)^{1/2}}\frac{w}{q}\frac{AR}{w}\left(1 + \frac{M}{fh}\right)(qT(AM + R))^\epsilon\\
    \kern-6pt&\ll_\epsilon&\kern-6pt (qT)^\epsilon\left(\frac{A}{qT}+\frac{AM}{fh(qT)}\right).
\end{eqnarray*}
Hence for 
$$\kappa \in \left(\frac{1}{2}, \frac{17}{33}\right), \ \frac{\log(T)}{\log(q)} < \frac{2\gamma + 1 - 2\kappa}{2\kappa - 1}$$
we have
\begin{equation}
\label{trivE}
\frac{\mathcal{E}_{w,f,h}(A,B,M,N)\phi(w)}{\phi^*(q)T} \ll_{\epsilon}(qT)^{\kappa + \epsilon - 1}+q^{\kappa -\frac{1}{2} -\gamma + \epsilon}T^{\kappa-\frac{1}{2}+\epsilon}
\end{equation}
which is $$O_\epsilon\left((qT)^{-\epsilon}\right).$$

When the trivial bound will not suffice, we use Mellin inversions to separate the variables in (\ref{E}) to reduce to finding a bound for
$$\frac{\mathcal{E}_{w,f,h}(A,B,M,N)\phi(w)}{\phi^*(q)T} \ll_\epsilon\sum_{d\geq 1} \frac{1}{(ABMN)^{1/2}}\frac{w(qT)^{\epsilon - 1}}{hf}$$
$$\times\left|\int\int_{x \asymp dM/B}\sum_{\substack{a \asymp A/d\\ (a,q)=1}}\sum_{\substack{b\asymp B/d\\ (b,aq)=1}}\sum_{\substack{0<|r|\leq R/wd\\0<|g|\leq Gfh/d}}\alpha_{ad}\beta_{bd}\nu_{rg}e\left(\frac{-gwr\overline{ah}}{bf} + \frac{gx}{hf}\right)dxdt\right|$$
where we may assume without loss of generality that $f\leq h$, otherwise we take Poisson summation modulo $ah$ instead of $bf$. We may also factor out $(w,h)$ from both $w$ and $h$, so that we can assume $w,h$ and $f$ are all pairwise co-prime, and all divide $q$, hence $whf\leq q$. By an adapted theorem of Bettin and Chandee from \cite{BetChan}, we arrive at the conclusion that the error
$$\frac{\mathcal{E}_{w,f,h}(A,B,M,N)\phi(w)}{\phi^*(q)T}$$
is at most 
\begin{eqnarray*}
\kern-6pt&\ll_\epsilon&\kern-6pt \sum_{d \geq 1} \frac{1}{(ABMN)^{1/2}}\frac{Md}{B}(qT)^\epsilon \frac{w}{qfh} ABd^{-2}\left(\frac{fh}{wT}\right)^{1/2}\left(1 + \frac{wABfhd^2}{fhABwTd^2}\right)^{1/4}\\
\kern-6pt&&\kern-6pt \times \left(h^{1/4}f^{1/2}\left(\frac{ABhf}{wTd^2}\right)(A/d)(B/d)^{3/4} + h^{3/4}f^{1/2}\left(\frac{ABhf}{wTd^2}\right)(A/d)^{1/2}(B/d)^{5/4} \right.\\
\kern-6pt& &\kern-6pt \left.+h^{-2/5}f^{6/5}\left(\frac{ABhf}{wTd^2}\right)(A/d)^{6/5}(B/d)^{1/10}+h^{1/5}f^{6/5}(A/d)^{2/5}(A/d)^{6/5}(B/d)^{7/10}\right.\\
\kern-6pt&&\kern-6pt \left.+h^{1/2}f^{3/5}\left(\frac{ABhf}{wTd^2}\right)^{7/10}(A/d)^{3/5}(B/d)^{13/10}+ h^{1/2}(A/d)(B/d)^{7/4}\right)^{1/2}\\
\kern-6pt&\ll_\epsilon&\kern-6pt  (qT)^\epsilon \left(\frac{h^{1/8}f^{1/4}AB^{7/8}}{qT}+\frac{h^{3/8}f^{1/4}A^{3/4}B^{9/8}}{qT}+ \frac{h^{-1/5}f^{3/5}A^{11/10}B^{11/20}}{qT}\right.\\
\kern-6pt& &\kern-6pt \left. +\frac{h^{-1/5}f^{3/10}w^{3/10}A^{4/5}B^{11/20}}{qT^{7/10}} + \frac{h^{1/10}f^{3/20}w^{3/20}A^{13/20}B}{qT^{17/20}} + \frac{h^{1/4}A^{1/2}B^{11/8}}{qT}\right)
\end{eqnarray*}
as $AM \asymp BN$.

For the first, second, and sixth terms substitute in $fh = q^\gamma$ and $f\leq h \Rightarrow f\leq q^{\gamma/2}$. For the third term, write $h^{-1/5} \leq f^{-1/5}$ then substitute in $f^{2/5} \leq q^{\gamma/5}$. For the fourth and fifth terms, we use the fact that $whf \leq q$. Hence the error is 
$$\ll_\epsilon (qT)^\epsilon \left((qT)^{\frac{15\kappa-8}{8}}q^{\frac{3\gamma}{16}}+ (qT)^{\frac{15\kappa-8}{8}}q^{\frac{3\gamma}{8}}+ (qT)^{\frac{33\kappa-20}{20}}q^{\frac{\gamma}{5}}+ (qT)^{\frac{27\kappa-14}{20}}+(qT)^{\frac{33\kappa-17}{20}}+(qT)^{\frac{15\kappa-8}{8}}q^{\frac{\gamma}{4}}\right).$$
 The first and sixth terms are smaller than the second, and the fourth is smaller than the fifth so 
 $$\frac{\mathcal{E}_{w,f,h}(A,B,M,N)\phi(w)}{\phi^*(q)T}\ll_\epsilon (qT)^{\frac{15\kappa -8}{8} + \frac{3\gamma}{8} + \epsilon}+(qT)^{\frac{33\kappa -20}{20} + \frac{\gamma}{5} + \epsilon} + (qT)^{\frac{33\kappa -17}{20}  + \epsilon} \ll_\epsilon (qT)^{-\epsilon}$$
 for 
 $$\kappa \in \left(\frac{1}{2}, \frac{17}{33}\right) \text{ and } \frac{\log(T)}{\log(q)} \geq \frac{2\gamma + 1 - 2\kappa}{2\kappa - 1}.$$
\end{proof}

\subsection{Proof of Theorem \ref{moment}}
In this section, we manipulate the main terms from the off-diagonals into a convenient form, then combine them with the diagonal terms.

We focus on the $\mathcal{M}^+$ terms first. Writing $W$ in terms of its Mellin transform, we see that
\begin{multline*}
\mathcal{M}^+_{w,f}(A,B,M,N) = \frac{1}{(2\pi i)^3}\int\int_{(\epsilon)}\int_{(c_2)}\int_{(c_1)}X_{+}(s,t)\tilde{W}(u)\tilde{W}(v)\left(\frac{q}{\pi}\right)^sM^uN^v\sum_{\substack{r \neq 0\\(r,f)=1}}\sum_{\substack{d,a,b\\ (a,b)=1\\(abd,q)=1}}\\
\frac{\alpha(A)_{da}\beta(B)_{db}}{a^{1+\beta + s+v}b^{1+\alpha+s+u}d}\int x^{-(1/2+\alpha + it + s + u)}(x-wr/ab)^{-(1/2+\beta - it + s + v)}dx\ \psi\left(\frac{t}{T}\right)dudv\frac{ds}{s}dt.
\end{multline*}

Now, we calculate the $x$ integral. If $r >0$ then the integral over $x$ is restricted to $x > wr/ab$ and if $r$ is negative then we have $x >0$. For absolute convergence, if $r >0$, we impose the condition
$$\operatorname{Re}(\alpha+\beta+ 2s + u + v) >0, \ \operatorname{Re}( \beta + s + v) <1/2$$
and if $r < 0$ we impose the condition 
$$\operatorname{Re}(\alpha+\beta  + 2s + u + v) >0, \ \operatorname{Re}( \alpha + s + u) <1/2.$$
Under these assumptions, the $x$-integral is equal to (see for example 17.43.21 and 17.43.22 of \cite{integral})
\begin{equation*}
 \left(\frac{w|r|}{ab}\right)^{-(\alpha+\beta + 2s + u + v)}\times \begin{cases}
\frac{\Gamma(\alpha +\beta + 2s + u + v)\Gamma(1/2-\beta  + it - s - v)}{\Gamma(1/2+\alpha +it + s + u)} & \text{if } r > 0\\
\frac{\Gamma(\alpha+\beta + 2s + u + v)\Gamma(1-/2-\alpha  - it - s - u)}{\Gamma(1/2+\beta -it + s + v)} & \text{if } r < 0
\end{cases}
\end{equation*}
and hence
$$\mathcal{M}^+_{w,f}(A,B,M,N) = \frac{1}{(2\pi i)^3}\sum_{\substack{r \neq 0\\(r,f)=1}}\sum_{\substack{d,a,b\\ (a,b)=1\\(abd,q)=1}}\frac{\alpha(A)_{ad}\beta(B)_{bd}}{d}\int\int_{(\epsilon)}\int_{(c_2)}\int_{(c_1)}X_{+}(s,t)$$
$$\tilde{W}(u)\tilde{W}(v)\left(\frac{q}{\pi}\right)^sw^{-(\alpha+\beta+2s+u+v)}a^{s+u-1+\alpha}b^{s+v-1+\beta}M^uN^v H_+(s)r^{-(\alpha+\beta + 2s + u + v)} dudv\frac{ds}{s}$$
where
$$H_+(s) = \Gamma(\alpha+\beta + 2s + u + v)\left(\frac{\Gamma(1/2-\beta  + it - s - v)}{\Gamma(1/2+\alpha +it + s + u)} + \frac{\Gamma(1/2-\alpha  - it - s - u)}{\Gamma(1/2+\beta -it + s + v)}\right).$$

In the $\mathcal{M}^-$ cases, due to the extra minus sign in the $x$-integral, we arrive at the same result but with $H_+(s)$ replaced by
$$H_-(s):= \frac{\Gamma(1/2-\alpha - it - s -u)\Gamma(1/2-\beta +it - s -v)}{\Gamma(1-\alpha-\beta - 2s - u - v)}.$$
Writing $H(s):= H_+(s) + H_-(s)$, and summing over $A,B \leq (qT)^\kappa$ in the dyadic decomposition allows us to write

\begin{equation}
\label{together}
 \frac{O^+}{\phi^*(q)T} = \frac{1}{\phi^*(q)T}\sum_{M,N}O^+(M,N) + O_\epsilon\left((qT)^{-\epsilon}\right)
\end{equation}

where 
\begin{multline*}
O^+(M,N) :=\frac{1}{2}\sum_{w|q}\mu\left(\frac{q}{w}\right)\phi(w)\sum_{\substack{f|q\\(f,w)}}\mu(f) \frac{\phi(q/f)}{q}\frac{1}{(2\pi i)^3}\sum_{\substack{r \geq 1\\(r,f)=1}}\sum_{\substack{d,a,b\\ (a,b)=1\\(abd,q)=1}}\frac{\alpha_{ad}\beta_{bd}}{d}\int\int_{(\epsilon)}\int_{(c_2)}\int_{(c_1)}\\
X_{+}(s,t)\tilde{W}(u)\tilde{W}(v)\left(\frac{q}{\pi}\right)^sw^{-(\alpha+\beta+2s+u+v)}a^{s+u-1+\alpha}b^{s+v-1+\beta}M^uN^v H(s)r^{-(\alpha+\beta + 2s + u + v)} dudv\frac{ds}{s}.
\end{multline*}

\begin{lemma}
$H(s)$ has simple poles at $s = 1/2 - \alpha - it - u$ and $s = 1/2 - \beta + it - v$, each of residue 2, and a zero at $s = \frac{1- \alpha - \beta  - u - v}{2}$.
\end{lemma}
\begin{proof}

Writing (for the sake of clarity) $x = \alpha + it + s + u$ and $y = \beta -it + s + v$, $H(s)$ is 
$$\frac{\Gamma(x+y)[\Gamma(1/2-y)\Gamma(1/2+y)+\Gamma(1/2-x)\Gamma(1/2+x)]}{\Gamma(1/2+x)\Gamma(1/2+y)}$$
$$+\frac{\Gamma(1/2-x)\Gamma(1/2-y)}{\Gamma(1-x-y)}$$
which has poles at $x,y = 1/2$ i.e. if $s = 1/2-\alpha - it - u$ or $1/2-\beta + it -v$. It is easy to check that these have residue 2. Also note that if $x + y  = 1$ then the second fraction vanishes (as there is a pole in the denominator from $\Gamma(1-x-y)$) and the first fraction is (using $\Gamma(s)\Gamma(1-s) = \pi/\sin(\pi s)$)
$$\frac{\Gamma(1)}{\Gamma(1/2 + x)\Gamma(1/2+y)}\left(\frac{\pi}{\sin(\pi(1/2+ y))}+ \frac{\pi}{\sin(\pi(1/2+x))}\right)=0$$
as $\sin(\pi(1/2 + y)) = \sin(\pi + \pi(1/2 -x)) = -\sin(\pi(1/2 - x)) = - \sin(\pi(1/2+x))$.
\end{proof}

Returning to $O^+(M,N)$ we move the contours to replace the $r$-sum with a zeta-function. 
Choose $c_1 = 0, \ c_2 =\epsilon$ and move the $s$-contour to the right to $1/2 -\epsilon/3$ crossing a simple pole of $H(s)$ at $s=1/2-\beta + it - v$. Write $P^{+ '}(M,N)$ as the integral along the new line and $R^+(M,N)$ as the residue. We can then move the $u$ contour in the residue to $\operatorname{Re}(u)=2\epsilon$ which hits no poles and allows us to replace the $r-$sum with a zeta function. i.e.
\begin{multline*}
R^+(M,N) = \frac{1}{2}\sum_{w|q}\mu\left(\frac{q}{w}\right)\phi(w)\sum_{\substack{f|q\\(f,w)}}\mu(f) \frac{\phi(q/f)}{q}\sum_{\substack{d,a,b\\ (a,b)=1\\(abd,q)=1}}\frac{\alpha_{ad}\beta_{bd}}{d}\frac{1}{(2\pi i)^2}\\
\int\int_{(\epsilon)}\int_{(2\epsilon)}X_{+}(1/2-\beta +it -v,t)
\tilde{W}(u)\tilde{W}(v)M^uN^v \\
\left(\frac{q}{\pi}\right)^{1/2-\beta +it - v}w^{-(\alpha-\beta + 1 + 2it + u - v)}a^{-1/2-\beta+it-v+u+\alpha}b^{-1/2 +it}\\
\prod_{p|f}\left(1-p^{-(\alpha-\beta + 2it + u - v+1)}\right)\zeta(\alpha-\beta + 1 + 2it + u - v)\psi\left(\frac{t}{T}\right)\frac{dudv}{1/2-\beta + it -v}dt.
\end{multline*}
Using the following lemma, we can simplify $R^+$ and $P^+$.
\begin{lemma} 
\label{zeta}

$$M := \sum_{w|q}\mu\left(\frac{q}{w}\right)\phi(w)\sum_{r \geq 1}\sum_{\substack{f|q\\(f,rw)=1}}\frac{\mu(f)\phi\left(\frac{q}{f}\right)}{q}w^{-s} \prod_{p|f}\left(1-p^{-s})\right) =  \phi^*(q)\prod_{p|q}\left(1-p^{s-1}\right)q^{-s}.$$

\end{lemma}

\begin{proof}
If $p|f$ and $p|q/f$ then $p^2|q$ but $(f,w)=1$ so $p^2|q/w\Rightarrow \mu(q/w) = 0$. Hence we may factorise $\phi(q/f)= \phi(q)/\phi(f)$, so
\begin{eqnarray*}
M \kern-6pt& =& \kern-6pt  \frac{\phi(q)}{q}\sum_{w|q}\mu\left(\frac{q}{w}\right)\phi(w)\sum_{\substack{f|q\\(f,w)=1}}\frac{\mu(f)}{\phi(f)}w^{-s}\prod_{p|f}\left(1-p^{-s}\right)\\
\kern-6pt& =& \kern-6pt \frac{\phi(q)}{q} q^{-s}\sum_{w|q}\mu(w)\phi\left(\frac{q}{w}\right)w^s\sum_{\substack{f|w\\(f,q/w)=1}}\frac{\mu(f)}{\phi(f)}\prod_{p|f}\left(1-p^{-s}\right)\\
\kern-6pt& =& \kern-6pt \frac{\phi(q)}{q} q^{-s}\sum_{w|q}\mu(w)\phi\left(\frac{q}{w}\right)w^s\prod_{\substack{p|w\\(p,q/w)=1}}\left(1- \frac{1-p^{-s}}{p-1}\right).
\end{eqnarray*}
Given that 
$$\phi\left(\frac{q}{w}\right) = \frac{q}{w}\prod_{p|q/w}\left(1-\frac{1}{p}\right)= \frac{\phi(q)}{w}\prod_{\substack{p|q\\(p,q/w)=1}}\left(1-\frac{1}{p}\right)^{-1} =\frac{\phi(q)}{w}\prod_{\substack{p|w\\(p,q/w)=1}}\left(1-\frac{1}{p}\right)^{-1} $$
we see that
\begin{eqnarray*}
M \kern-6pt& =& \kern-6pt \frac{\phi(q)^2}{q} q^{-s}\sum_{w|q}\mu(w)w^{s-1}\prod_{\substack{p|w\\(p,q/w)=1}}\left(1- \frac{1-p^{-s}}{p-1}\right)\left(1-\frac{1}{p}\right)^{-1}\\
\kern-6pt& =& \kern-6pt \frac{\phi(q)^2}{q} q^{-s}\sum_{w|q}\mu(w)w^{s-1}\prod_{\substack{p|w\\(p,q/w)=1}}\left(\frac{p^2-2p+p^{1-s}}{(p-1)^2}\right)\\
\kern-6pt& =& \kern-6pt \frac{\phi(q)^2}{q} q^{-s}\prod_{p^2|q}\left(1-p^{s-1}\right)\prod_{p||q}\left(1- p^{s-1}\left(\frac{p^2-2p+p^{1-s}}{(p-1)^2}\right)\right)
\end{eqnarray*}
as if $p^2|q$ and $p|w$ then either $\mu(w)=0$ or $p|(w,q/w)$ so either the sum is empty (i.e. equal to zero) or the product is empty (equal to 1). Then rearranging gives that 
\begin{eqnarray*}
M \kern-6pt& =& \kern-6pt \frac{\phi(q)^2}{q} q^{-s}\prod_{p^2|q}\left(1-p^{s-1}\right)\prod_{p||q}\frac{\left(1-\frac{2}{p}\right)}{\left(1-\frac{1}{p}\right)^2}\left(1-p^{s-1}\right)\\
\kern-6pt& =& \kern-6pt
\phi^*(q) \prod_{p|q}\left(1-p^{s-1}\right)q^{-s}.
\end{eqnarray*}
as $\phi^*(q) = q\prod_{p||q}\left(1-\frac{2}{p}\right)\prod_{p^2|q}\left(1-\frac{1}{p}\right)^2$.

\end{proof}
By Lemma \ref{zeta},
$$R^+(M,N) = \frac{\phi^*(q)}{2}\sum_{\substack{d,a,b\\ (a,b)=1\\(abd,q)=1}}\frac{\alpha_{ad}\beta_{bd}}{d}\frac{1}{(2\pi i)^2}\int\int_{(\epsilon)}\int_{(2\epsilon)}X_{+}(1/2-\beta +it -v,t)$$
$$\tilde{W}(u)\tilde{W}(v)M^uN^v \left(\frac{q}{\pi}\right)^{1/2-\beta +it - v}q^{-(\alpha-\beta + 1 + 2it + u - v)}a^{-1/2-\beta+it-v+u+\alpha}b^{-1/2 +it}$$
$$\prod_{p|q}\left(1-p^{\alpha-\beta + 2it + u - v}\right)\zeta(\alpha-\beta + 1 + 2it + u - v)\psi\left(\frac{t}{T}\right)\frac{dudv}{1/2-\beta + it -v}dt.$$
With the $P^{+ '}(M,N)$ term, we replace the $r$-sum with a zeta-function as before, apply Lemma \ref{zeta} and shift the $s$-contour back to $\operatorname{Re}(s) = \epsilon$. This crosses the same pole at $s = 1/2-\beta + it -v$, while the pole from the zeta function at $s = (1-\alpha -\beta -u-v)/2$ is cancelled out by the zero of $H(s)$ at this point. Denote the contribution from the first pole as $R^{+ '}(M,N)$ and the new integral with the $r$-sum replaced as $P^{+ ''}(M,N)$ so
$$O^{+}(M,N)  = P^{+ ''}(M,N) + R^{+ '}(M,N) - R^+(M,N)$$
where $P^{+ ''}(M,N)$ is
$$\frac{\phi^*(q)}{2(2\pi i)^3}\sum_{\substack{d,a,b\\ (a,b)=1\\(abd,q)=1}}\frac{\alpha_{ad}\beta_{bd}}{d}\int\int_{(\epsilon)}\int_{(\epsilon)}\int_{(0)}X_{+}(s,t)\tilde{W}(u)\tilde{W}(v)\left(\frac{q}{\pi}\right)^sq^{-(\alpha+\beta+2s+u+v)}$$
$$a^{s+u-1+\alpha}b^{s+v-1+\beta}M^uN^v H(s)\prod_{p|q}\left(1-p^{\alpha+\beta-1+2s+u+v}\right)\zeta(\alpha+\beta + 2s + u + v) \psi\left(\frac{t}{T}\right)dudv\frac{ds}{s}dt.$$
The difference between the two residue terms is in the $u$-contour i.e. integrating over $\operatorname{Re}(u) = 0, 2\epsilon$. Therefore $R^{+ \prime}(M,N) - R^+(M,N)$ is the residue at $u = \beta-\alpha + v - 2it$ but this is cancelled by the zero from the $(1-p^{\alpha-\beta + 2it+u-v})$ factors,w i.e. $R^{+ \prime}(M,N) - R^+(M,N) = 0$. Hence $O^+(M,N) = P^{+\prime\prime}(M,N)$.

By Lemma 4.3 in \cite{Bui2020} we can remove the dyadic partition  i.e.
$$ O_0^+ :=\sum_{M,N}P^{+ ''}(M,N) = \frac{\phi^*(q)}{4\pi i}\sum_{\substack{d,a,b\\ (a,b)=1\\(abd,q)=1}}\frac{\alpha_{ad}\beta_{bd}}{d}\int\int_{(\epsilon)}X_{+}(s,t)\left(\frac{q}{\pi}\right)^s$$
$$q^{-\alpha-\beta-2s}a^{s-1+\alpha}b^{s-1+\beta}H(s)\prod_{p|q}\left(1-p^{\alpha+\beta -1+2s}\right)\zeta(\alpha+\beta + 2s )\psi\left(\frac{t}{T}\right)\frac{ds}{s}dt.$$
We can now write $H(s)$ as 
$$H(s) = \frac{\Gamma(\alpha + \beta + 2s)\Gamma(1/2-\beta + it - s)}{\Gamma(1/2+\alpha + it + s)}+\frac{\Gamma(\alpha+\beta + 2s)\Gamma(1/2-\alpha - it - s)}{\Gamma(1/2+\beta - it + s)}$$
$$+\frac{\Gamma(1/2-\alpha -it-s)\Gamma(1/2-\beta + it - s)}{\Gamma(1-\alpha-\beta -2s)}$$
which by Lemma 8.2 of \cite{Young} is equal to 
$$\pi^{1/2}\frac{\Gamma(\frac{\alpha+\beta+2s}{2})\Gamma(\frac{1/2-\alpha -it -s}{2})\Gamma(\frac{1/2-\beta +it -s}{2})}{\Gamma(\frac{1-\alpha-\beta -2s}{2})\Gamma(\frac{1/2+\alpha + it +s}{2})\Gamma(\frac{1/2+\beta - it +s}{2})}.$$
Hence, 
\begin{equation}
\label{H}
H(s)X_+(s,t) = \pi^{1/2}X_-(s,t)\frac{\Gamma\left(\frac{\alpha + \beta + 2s}{2}\right)}{\Gamma\left(\frac{1-\alpha-\beta-2s}{2}\right)}.
\end{equation}
Applying the functional equation 
$$\pi^{-(\alpha+\beta + 2s)/2}\Gamma\left(\frac{\alpha+\beta + 2s}{2}\right)\zeta(\alpha+\beta + 2s)$$
$$= \pi^{-(1-\alpha-\beta - 2s)/2}\Gamma\left(\frac{1-\alpha-\beta - 2s}{2}\right)\zeta(1-\alpha-\beta - 2s)$$
and the change of variable $s \rightarrow -s$ gives $O_0^+$ as
\begin{eqnarray*}
\kern-6pt&=&\kern-6pt -\frac{\phi^*(q)}{2}\left(\frac{q}{\pi}\right)^{-\alpha-\beta}\sum_{\substack{d,a,b\\ (a,b)=1\\(abd,q)=1}}\frac{\alpha_{ad}\beta_{bd}}{a^{1-\alpha}b^{1-\beta}d}\frac{1}{2\pi i}\int\int_{(-\epsilon)} X_{-}(s,t)\left(\frac{\pi ab}{q}\right)^{-s}\prod_{p|q}\left(1-p^{\alpha+\beta-1-2s}\right)\\
\kern-6pt& &\kern-6pt\zeta(1-\alpha-\beta + 2s)\psi\left(\frac{t}{T}\right)\frac{ds}{s}dt\\
\kern-6pt&=&\kern-6pt -\frac{\phi^*(q)}{2}\left(\frac{q}{\pi}\right)^{-\alpha-\beta}\sum_{\substack{d,a,b\\ (a,b)=1\\(abd,q)=1}}\frac{\alpha_{ad}\beta_{bd}}{a^{1-\alpha} b^{1-\beta}d}\frac{1}{2\pi i}\int\int_{(-\epsilon)} X_{-}(s,t)\left(\frac{\pi ab}{q}\right)^{-s}\\
\kern-6pt&&\kern-6ptL(1-\alpha-\beta + 2s, \chi_{0,q})\psi\left(\frac{t}{T}\right)\frac{ds}{s}dt.
\end{eqnarray*}
To summarise:
$$\frac{O^+}{\phi^*(q)T} = \frac{O_0^+}{\phi^*(q)T}+O_\epsilon\left((qT)^{-\epsilon}\right).$$
The $O^-$ case is identical by replacing $X_+$ with $X_-$ and the substitution $\alpha,\beta \rightarrow -\beta, -\alpha$. 
\subsection{Combining the Main Terms}
We have shown that
$$\frac{1}{\phi^*(q)T}\int\sideset{}{^+}\sum_{\chi\Mod{q}}L(1/2 + \alpha + it, \chi)L(1/2 + \beta - it, \bar{\chi})\sum_{a,b \leq (qT)^\kappa}\frac{\alpha_a\beta_b\chi(a)\bar{\chi}(b)}{a^{1/2 +it}b^{1/2 - it}}\psi(t/T)dt$$ 
$$= \frac{1}{\phi^*(q)T}\left(D^+ + O_0^+ + \left(\frac{q}{\pi}\right)^{-\alpha-\beta}\left(D^- + O_0^-\right)\right)+  O_\epsilon\left((qT)^{(\frac{33\kappa - 17}{20}+\epsilon)}\right)$$ 
where for instance
\begin{eqnarray*}
\left(\frac{q}{\pi}\right)^{-\alpha-\beta}D^- + O_0^+ \kern-6pt& = &\kern -6pt \left(\frac{q}{\pi}\right)^{-\alpha-\beta}\frac{\phi^*(q)}{2}\sum_{\substack{ad,bd \leq (qT)^\kappa\\ (a,b)=1\\(abd,q)=1}} \frac{\alpha_{ad}\beta_{bd}}{a^{1-\alpha}b^{1-\beta}d}\frac{1}{2\pi i}\\
\kern-6pt& &\kern-6pt \left(\int_{(\epsilon)}- \int_{(-\epsilon)}\right)\int X_-(s,t)\left(\frac{q}{ab\pi}\right)^sL(1-\alpha-\beta + 2s,\chi_{0,q})\psi\left(\frac{t}{T}\right)dt\frac{ds}{s}\\
\kern-6pt& = &\kern -6pt \text{Res}_{s=0}\\
\kern-6pt& = &\kern -6pt \frac{\phi^*(q)}{2}\left(\frac{q}{\pi}\right)^{-\alpha-\beta}L(1-\alpha-\beta, \chi_{0,q})\sum_{\substack{ad,bd \leq (qT)^\kappa\\ (a,b)=1\\(abd,q)=1}} \frac{\alpha_{ad}\beta_{bd}}{a^{1-\alpha}b^{1-\beta}d}\int X_{-}(0,t)\psi\left(\frac{t}{T}\right)dt.
\end{eqnarray*}
Note that the pole at $s = (\alpha+\beta)/2$ of the $L$-function is cancelled by the function $G$. A similar expression holds for the sum of the other two terms, giving the result in Theorem \ref{moment} for the sum over even Dirichlet characters.

\subsection{The Odd Characters}
\label{odd}
The odd characters go through almost identically, but with two differences: firstly we have to redefine the functions $V_{\pm}(.)$ (because of the different functional equation for odd Dirichlet characters) by altering the gamma functions, and secondly summing over the odd primitive characters gives a different sum to the even characters i.e. for $(m,q)=1$,
$$\sideset{}{^-}\sum_{\chi \Mod{q}}\chi(m) = \frac{1}{2}\sum_{\substack{uw=q\\ m \equiv + 1 \Mod{w}}}\mu(u)\phi(w) - \frac{1}{2}\sum_{\substack{uw=q\\ m \equiv - 1 \Mod{w}}}\mu(u)\phi(w).$$
This manifests itself in our definition of the function $H(s)$ at  (\ref{together}). In this setting we must redefine $H(s):= H_+(s) - H_-(s)$. The same method still works as our new $H(s)$ has zeros in the same positions, and no poles so there are not any residue terms to deal with. To show at (\ref{H}) that 
$$H(s)X_+(s,t) = \pi^{1/2}X_-(s,t)\frac{\Gamma\left(\frac{\alpha + \beta + 2s}{2}\right)}{\Gamma\left(\frac{1-\alpha-\beta-2s}{2}\right)}$$
we appeal to Lemma 8.4 of \cite{Young} instead of Lemma 8.2. Hence
\begin{multline*}
\frac{1}{\phi^*(q)T}\int\sideset{}{^-}\sum_{\chi\Mod{q}}L(1/2 + \alpha + it, \chi)L(1/2 + \beta - it, \bar{\chi})\sum_{a,b \leq (qT)^\kappa}\frac{\alpha_a\beta_b\chi(a)\bar{\chi}(b)}{a^{1/2 + it}b^{1/2  - it}}\psi\left(\frac{t}{T}\right)dt\\
= \frac{\hat{\psi}(0)}{2} L(1 + \alpha + \beta, \chi_{0,q})\sum_{\substack{ad,bd \leq (qT)^\kappa\\ (a,b)=1\\(abd,q)=1}} \frac{\alpha_{ad}\beta_{bd}}{a^{1+\beta}b^{1+\alpha}d} + \frac{1}{2T}\left(\frac{q}{\pi}\right)^{-\alpha-\beta}L(1-\alpha - \beta, \chi_{0,q})\\
\times\sum_{\substack{ad,bd \leq (qT)^\kappa\\ (a,b)=1\\(abd,q)=1}} \frac{\alpha_{ad}\beta_{bd}}{a^{1-\alpha}b^{1-\beta}d}\int\frac{\Gamma\left(\frac{3/2-\alpha-it}{2}\right)\Gamma\left(\frac{3/2-\beta+it}{2}\right)}{\Gamma\left(\frac{3/2+\alpha+it}{2}\right)\Gamma\left(\frac{3/2+\beta-it}{2}\right)}\psi\left(\frac{t}{T}\right)dt +O\left((qT)^{-\epsilon}\right).
\end{multline*}

\subsection{Proof of Theorem \ref{moment2}}
This proof is the same as Theorem \ref{moment}, except that we use the Vaughan identity with the M\"obius function to split up $\mathcal{E}_{w,f,h}$ in (\ref{E}) into three sums, which are then bounded separately. 
\subsubsection{The Vaughan Identity}
Let $U(s) = \sum_{n \leq W}\mu(n)n^{-s}$ for $W$ a constant that we shall choose later. By comparing the coefficients of 
$$\frac{1}{\zeta(s)} = \frac{1}{\zeta(s)}(1-\zeta(s)U(s))^2 + 2U(s) - \zeta(s)U(s)^2$$
we see that 
$$\mu(u) = c_1(u)+c_2(u)+c_3(u)$$
where 
$$c_1(u) = \sum_{\substack{abc = u\\a \geq W, b \geq W}}\mu(c)c_4(a)c_4(b) \text{ with } \ c_4(a) = -\sum_{\substack{ef=a\\e \leq W}}\mu(e)$$
$$c_2(u) = \begin{cases} 2\mu(u) & \text{if } u \leq W\\
0 & \text{if } u > W
\end{cases}$$
$$c_3(u) = - \sum_{\substack{abc = u\\ a \leq W, b \leq W}}\mu(a)\mu(b).$$
Substituting $\alpha_{ad}(A) = \mu(ad)f_A(ad) = c_1(ad)f_A(ad)+c_2(ad)f_A(ad)+c_3(ad)f_A(ad)$ into (\ref{E}) (where $f_A$ is $f$ multiplied by a smooth function supported on $[A,2A]$, so that $f_A^{'}(x) \ll_\epsilon x^{-1 + \epsilon}$) produces
$$\mathcal{E}_{w,f,h}(A,B,M,N) = E_1(A,B,M,N) + E_2(A,B,M,N) + E_3(A,B,M,N)$$
where (using Mellin transforms to separate variables)
$$\frac{{E}_{i}(A,B,M,N)\phi(w)}{\phi^*(q)T} \ll_\epsilon\sum_{d\geq 1} \frac{1}{(ABMN)^{1/2}}\frac{w(qT)^{\epsilon - 1}}{hf}$$
$$\times\left|\int\int_{x \asymp dM/B}\sum_{\substack{a \asymp A/d\\ (a,q)=1}}\sum_{\substack{b\asymp B/d\\ (b,aq)=1}}\sum_{\substack{0<|r|\leq R/wd\\0<|g|\leq Gfh/d}}c_i(ad)f_A(ad)\beta_{bd}\nu_{rg}e\left(\frac{-gwr\overline{ah}}{bf}+ \frac{gx}{hf}\right)dxdt\right|.$$
Let $W = A^{1/4}$. This means that $E_2(A,B,M,N)$ is an empty sum as the sequence $\alpha_n(A)$ has support on $[A,2A] \cap [1, A^{1/4}]$.
\subsubsection{$E_1$}
To bound $ \frac{{E}_{1}(A,B,M,N)\phi(w)}{\phi^*(q)T}$ we write it as a linear combination of at most $O_\epsilon\left((qT)^\epsilon\right)$ sums, each of which is
\begin{multline}
    \label{sum}
    \ll_\epsilon \sum_{\substack{d\geq 1\\d_1d_2d_3=d}} \frac{1}{(ABMN)^{1/2}}\frac{w(qT)^{\epsilon - 1}}{hf} \bigg|\int\int_{x \asymp dM/B}\sum_{\substack{a_1 \asymp A_1/d_1\\a_2 \asymp A_2/d_2\\a_3 \asymp A_3/d_3\\ (a_1a_2a_3,q)=1}}\sum_{\substack{b\asymp B/d\\ (b,a_1a_2a_3q)=1}}\sum_{\substack{0<|r|\leq R/wd\\0<|g|\leq Gfh/d}}c_4(a_1d_1)\\
    c_4(a_2d_2)\mu(a_3d_3)f_A(a_1a_2a_3d)\beta_{bd}\nu_{rg}e\left(\frac{-gwr\overline{a_1a_2a_3h}}{bf}+ \frac{gx}{hf}\right)dxdt\bigg|
\end{multline}
with $A_1A_2A_3 = A$ and where we may assume that $a_1,a_2,a_3,d$ are all pairwise coprime and square-free due to the presence of the M\"obius function. By the definition of $c_4$ we see that $A_1, A_2 \gg W/d$ and without loss of generality $A_1 \leq A_2$. By defining 
$$c_5(a_2'd_2') = \sum_{\substack{d_2d_2=d_2'\\a_2a_3=a_2'}}\mu(a_3d_3)c_4(a_2d_2)$$

we change (\ref{sum}) into sums of the form

\begin{multline}
    \label{sum1}
\ll_\epsilon \sum_{\substack{d\geq 1\\d_1d_2=d}} \frac{1}{(ABMN)^{1/2}}\frac{w(qT)^{\epsilon - 1}}{hf}\bigg|\int\int_{x \asymp dM/B}\sum_{\substack{a_1 \asymp A_1/d_1\\a_2 \asymp A_2/d_2\\ (a_1a_2,q)=1}}\sum_{\substack{b\asymp B/d\\ (b,a_1a_2q)=1}}\sum_{\substack{0<|r|\leq R/wd\\0<|g|\leq Gfh/d}}c_4(a_1d_1)\\
c_5(a_2d_2)f_A(a_1a_2d)\beta_{bd}\nu_{rg}e\left(\frac{-gwr\overline{a_1a_2h}}{bf}+ \frac{gx}{hf}\right)dxdt\bigg|
\end{multline}
with $W \ll A_1 \ll A_2 \ll A/W$ and $A_1A_2 = A$. Let $A_1 = (qT)^{\kappa_1}$ and $A_2 = (qT)^{\kappa_2}$ so that $\kappa = \kappa_1 + \kappa_2$. Note that we may bound $c_4(n), c_5(n)$ by $O_\epsilon((qT)^\epsilon)$.
By applying Lemma 6 of \cite{RHB} (slightly adapted to include the extra $h,f$ in the trilinear fraction) with 
\begin{itemize}
\item $U \leftrightarrow \frac{Bf}{d}$
\item $K \leftrightarrow \frac{RGfh}{wd^2} \asymp ABfh(wT)^{\epsilon - 1}d^{-2}$
\item $S \leftrightarrow \frac{A_1h}{d_1}$
\item $T \leftrightarrow \frac{A_2}{d_2}$

\end{itemize}
to bound the sums in (\ref{sum1}) by 
\begin{eqnarray*}
\kern-6pt & \ll_\epsilon & \kern-6pt \sum_{\substack{d\geq 1\\d_1d_2=d}} \frac{1}{(ABMN)^{1/2}}\frac{w(qT)^{\epsilon - 1}}{hf}T\frac{dM}{B} \frac{B}{d}\frac{ABhf}{wTd^2}\frac{A_1}{d_1}\frac{A_2}{d_2}\left(\left(\frac{fd_1d^2wT}{AA_1Bhf}\right)^{1/4}\right.\\
\kern-6pt& &\kern-6pt + \left.\left(\frac{BfwTd^2d_1d_2^2}{dABhfA_1A^2_2}\right)^{1/4} + \frac{d^{1/4}}{B^{1/4}}+ \frac{d_2^{1/2}}{A_2^{1/2}}\right)\\
\kern-6pt & \ll_\epsilon & \kern-6pt AB(qT)^{\epsilon-1}\left(\left(\frac{qT}{A_1AB}\right)^{1/4} + \left(\frac{qT}{A_2A^2}\right)^{1/4} + B^{-1/4} + A_2^{-1/2}\right)\\
\kern-6pt & \ll_\epsilon & \kern-6pt (qT)^{\frac{3\kappa}{2} - \frac{\kappa_1}{4} - \frac{3}{4}} + (qT)^{\frac{3\kappa}{2}-\frac{\kappa_2}{4} - \frac{3}{4}} + (qT)^{\frac{7\kappa}{4} - 1}+ (qT)^{\frac{3\kappa}{2} + \frac{\kappa_1}{2}-1}\\
\kern-6pt & \ll_\epsilon & \kern-6pt (qT)^{\frac{3\kappa}{2} - \frac{\kappa_1}{4} - \frac{3}{4}}
\end{eqnarray*}
for $\kappa < 4/7$ and $\kappa_1 \leq \kappa/2$. This bound is less effective when $A_1$ is small, so another bound is needed in this case. Using lemmas 10 and 11 from \cite{mato} with  
\begin{itemize}
\item $C \leftrightarrow B/d$
\item $M \leftrightarrow A_2/d_2$
\item $K \leftrightarrow RGhf/wd^2 \ll ABhf(wT)^{\epsilon - 1}d^{-2}$
\item $R \leftrightarrow A_1h/d_1$
\item $ d \leftrightarrow w$
\item $ s \leftrightarrow f$
\item $X_d \leftrightarrow T^{\epsilon-1/2}$
\end{itemize}
we may bound the sums in (\ref{sum1}) by 
\begin{eqnarray*}
\kern-6pt & \ll_\epsilon & \kern-6pt \sum_{\substack{d\geq 1\\d_1d_2=d}} \frac{1}{(ABMN)^{1/2}}\frac{w(qT)^{\epsilon - 1}}{hf}T\frac{dM}{B}\frac{AB}{d^2}\sqrt{\frac{hf}{wT}}\left[\frac{A_2B}{dd_2}\right.\\
\kern-6pt& &\kern-6pt + \frac{A^{3/2}B^{1/2}}{d^2f}\sqrt{\frac{hf}{wT}}\left(\frac{ABhf}{wTd^2} + \frac{A_1}{d_1}\right)^{1/2}+ (wT)^{7/64}\frac{A_2^{1/2}B^{1/2}h^{1/2}f^{1/2}A_1^{3/2}}{d^{1/2}d_1^{3/2}d_2^{1/2}}\left(\frac{ABhf}{wTd^2} + \frac{A_1}{d_1}\right)^{1/2}\\
\kern-6pt& & \kern-6pt \left.\times\left(1+ \frac{B^{1/2}d_1d_2^{1/2}}{d^{1/2}A_1A_2^{1/2}h^{1/2}}\right)\left(1+\frac{ABhfd_1}{wTd^2A_1hf}\right)^{1/2}\right]\\
\kern-6pt & \ll_\epsilon & \kern-6pt (qT)^{\frac{\kappa + \kappa_2 - 1}{2}}q^{-\frac{\gamma}{2}}+ (qT)^{\frac{3\kappa}{2}-1}f^{-\frac{1}{2}}+(qT)^{\kappa + \frac{\kappa_1}{4}-\frac{3}{4}}f^{-\frac{1}{2}}q^{-\frac{\gamma}{4}} \\
\kern-6pt &  & \kern-6pt + \frac{(wT)^{\frac{7}{128}}}{qT^{1/2}}\left(\frac{(qT)^{\kappa + \frac{\kappa_1}{2}}}{T^{\frac{1}{4}}}\left(\frac{w}{hf}\right)^{\frac{1}{4}}q^{\frac{\gamma}{4}} + (qT)^{\frac{\kappa}{2}+\frac{3\kappa_1}{4}}w^{\frac{1}{2}}q^{-\frac{\gamma}{4}}\right)\left(1 + \frac{(qT)^{\frac{\kappa + \kappa_2}{4}}}{(wT)^{1/4}}\right)\\
\kern-6pt & \ll_\epsilon & \kern-6pt (qT)^{\kappa - \frac{\kappa_1}{2}-\frac{1}{2}} + (qT)^{\frac{3\kappa}{2}-1} + (qT)^{\kappa + \frac{\kappa_1}{4} - \frac{3}{4}} + (qT)^{\kappa + \frac{\kappa_{1}}{2} - \frac{89}{128}}\\
\kern-6pt& & \kern-6pt  +(qT)^{\frac{\kappa}{2} + \frac{3\kappa_1}{4} - \frac{57}{128}} + (qT)^{\frac{3\kappa}{2} + \frac{\kappa_1}{4} -\frac{121}{128}} + (qT)^{\kappa + \frac{\kappa_1}{2} - \frac{89}{128}}\\
\kern-6pt & \ll_\epsilon & \kern-6pt (qT)^{\kappa - \frac{\kappa_1}{2}-\frac{1}{2}} +(qT)^{\kappa + \frac{\kappa_{1}}{2} - \frac{89}{128}} +(qT)^{\frac{\kappa}{2} + \frac{3\kappa_1}{4} - \frac{57}{128}}
\end{eqnarray*}
for $\kappa< 1/2 + 5/128$ and $\kappa_1 < \kappa/2$. We use the first bound when $\kappa_1 \leq \kappa - \frac{39}{128}$ and the second bound for when $\kappa_1 \geq \kappa-\frac{39}{128}$, resulting in the bound
$$ \frac{{E}_{1}(A,B,M,N)\phi(w)}{\phi^*(q)T} \ll_\epsilon (qT)^{\frac{5\kappa}{4} - \frac{345}{512}+\epsilon} \ll_\epsilon (qT)^{-\epsilon}.$$

\subsubsection{$E_3$}
$\frac{{E}_{3}(A,B,M,N)\phi(w)}{\phi^*(q)T}$ may be bounded by a sum of at most $O_\epsilon((qT)^\epsilon)$ sums of the form
\begin{equation*}
    \begin{split}
        \ll_\epsilon &\sum_{\substack{d\geq 1\\d_1d_2=d}} \frac{1}{(ABMN)^{1/2}}\frac{w(qT)^{\epsilon - 1}}{hf}\\
        &\times\bigg|\int\int_{x \asymp dM/B}\sum_{\substack{a_1 \asymp A_1/d_1\\a_2 \asymp A_2/d_2\\ (a_1a_2,q)=1}}\sum_{\substack{b\asymp B/d\\ (b,a_1a_2q)=1}}\sum_{\substack{0<|r|\leq R/wd\\0<|g|\leq Gfh/d}}c_6(a_1d_1)f_A(a_1a_2d)\beta_{bd}\nu_{rg}e\left(\frac{-gwr\overline{a_1a_2h}}{bf}+ \frac{gx}{hf}\right)dxdt\bigg|
    \end{split}
\end{equation*}

with 
$$c_6(n) = \sum_{\substack{xy = n\\ x,y \leq W}}\mu(x)\mu(y).$$
This means that $A_1 \leq W^2 = A^{1/2}$. When $A_1 \gg A^{1/4}$ we use the same method as for $E_1$, but when $A_1 \ll A^{1/4}$ we shall apply the Weil bound for Kloosterman sums. This implies that 
$$\sum_{\substack{A_2/d_2 \leq a_2 \leq A_2/d_2+x\\(a_2,bq)=1}}e\left(\frac{-gwr\overline{a_1a_2h}}{bf}\right)\ll_\epsilon (bf)^{1/2 + \epsilon}(gwr\overline{a_1h},bf)\left(1+ \frac{A_2}{bfd_2}\right).$$
By partial summation over $a_2$ we may bound the sums above by

\begin{eqnarray*}
\kern-6pt & \ll_\epsilon & \kern-6pt \sum_{\substack{d\geq 1\\d_1d_2=d}} \frac{1}{(ABMN)^{1/2}}\frac{w(qT)^{\epsilon - 1}}{hf}T\frac{dM}{B} (qT)^\epsilon \frac{A_1}{d_1}\left(\frac{Bf}{d_2}\right)^{1/2}\left(1+ \frac{A_2d}{Bfd_2}\right)\sum_{\substack{0 < |r| \leq R/wd\\ (r,f)=1\\0 < |g| \leq Gfh/d}}\sum_{b \asymp B/d}(rg,bf)\\
\kern-6pt & \ll_\epsilon & \kern-6pt (qT)^{\epsilon-1} A_1B^{3/2}f^{1/2}\left(1+ \frac{A_2}{Bf}\right)\\
\kern-6pt & \ll_\epsilon & \kern-6pt f^{1/2}(qT)^{\frac{3\kappa}{2} + \kappa_1 + \epsilon - 1} \\
\kern-6pt & \ll_\epsilon & \kern-6pt q^{\gamma/2}(qT)^{\frac{7\kappa}{4} + \epsilon - 1}.
\end{eqnarray*}
By the trivial bound in Lemma \ref{triv} and (\ref{trivE}), we may assume that for $1/2 < \kappa < 1/2 + 5/128$
$$\frac{\log(T)}{\log(q)} \geq \frac{2\gamma + 1 - 2\kappa}{2\kappa - 1}$$
which means that 
\begin{eqnarray*}
\frac{{E}_{3}(A,B,M,N)\phi(w)}{\phi^*(q)T} \kern-6pt & \ll_\epsilon & \kern-6pt q^{\frac{\gamma}{2} + \left(\frac{7\kappa}{4} - 1\right)\left(1 + \frac{2\gamma + 1 - 2\kappa}{2\kappa - 1}\right)+ \epsilon}\\
\kern-6pt & \ll_\epsilon & \kern-6pt q^{\gamma\left(\frac{9\kappa - 5}{4\kappa - 2}\right)+\epsilon}\\
\kern-6pt & \ll_\epsilon & \kern-6pt (qT)^{-\epsilon}.
\end{eqnarray*}

This concludes the proof of Theorem \ref{moment2}.

\section{Proof of Theorem \ref{theoremdensity}}
\label{Density proof}

To prove Theorem \ref{theoremdensity}, first note that for $0\leq \sigma -\frac{1}{2}\leq \frac{1}{\log(qT)}$, 
$$N(\sigma,T,\chi) \leq N(1/2,T,\chi) <T\log(qT).$$
Also, for $\sigma -\frac{1}{2} \geq \frac{28\log\log(qT)}{\log(qT)}$, the theorem is true by Montgomery's result. So it is sufficient to prove the following proposition 
\begin{proposition}
\label{density}
For $\frac{1}{\log(qT)} \leq \sigma -\frac{1}{2} \leq  \frac{28 \log\log(qT)}{\log(qT)}$ and $\kappa < 1/2 + 5/128$, 

$$\sideset{}{^*}\sum_{\chi\Mod{q}}N(\sigma,T;\chi)\ll(qT)^{2-2\sigma}\log^5(qT) + (2\sigma -1 )(qT)^{1+\kappa(1-2\sigma)}\log(qT)^3.$$
\end{proposition}

To prove this proposition, we rely on Littlewood's lemma (see \cite{Titchmarsh} Theorem 9.16), which reduces the problem of bounding 
$$\int_{T}^{2T}\sideset{}{^*}\sum_{\chi\Mod{q}}|L(\sigma + it,\chi)M(\sigma+it,\chi)-1|^2\psi\left(\frac{t}{T}\right)dt$$ 
by 
 $$O\left((qT)^{2-2\sigma}\log^4(qT) + (2\sigma -1 )(qT)^{1+\kappa(1-2\sigma)}\log(qT)^2 \right)$$ 
where $\psi(t)$ is a smoothing function as in Theorem \ref{moment}. By expanding out the square in the integral, we get three terms
\begin{equation}
\label{2nd}
\int_{T}^{2T}\sideset{}{^*}\sum_{\chi\Mod{q}}|L(\sigma + it,\chi)M(\sigma+it,\chi)|^2\psi\left(\frac{t}{T}\right)dt
\end{equation}
$$- 2 \operatorname{Re}\left(\int_{T}^{2T}\sideset{}{^*}\sum_{\chi\Mod{q}}L(\sigma + it,\chi)M(\sigma+it,\chi)\psi\left(\frac{t}{T}\right)dt\right)
 + \phi^*(q)T\hat{\psi}(0).$$

We look first at the term (\ref{2nd}). Using methods similar to those used by Iwaniec and Sarnak in \cite{sarnak}, if our mollifier is of the form
$$M(s,\chi)= \sum_{n \leq x}\frac{v(n)\chi(n)}{n^{\sigma}}$$ 
then the optimal mollifier (with the normalisation that $v(1) =1$) can be shown to be close to
$$v(n) = \frac{\mu(n)(1-(x/n)^{1-2\sigma})}{1-(x)^{1-2\sigma}}$$
for $1 \leq n \leq x$ and 0 otherwise. Note that  
$$\lim_{\sigma \rightarrow 1/2}v(n)= \frac{\mu(n)\log(x/n)}{\log(x)}$$
which is a standard mollifier on the half-line. This choice of mollifier satisfies the conditions of Theorem \ref{moment2} and so by defining 
$$S_1(x) := \sum_{\substack{a,b,d\\(a,b)=1\\(abd,q)=1}}\frac{v(ad)v(bd)}{(abd)^{2\sigma}}$$
and 
$$S_2(x) := \sum_{\substack{a,b,d\\(a,b)=1\\(abd,q)=1}}\frac{v(ad)v(bd)}{abd^{2\sigma}}$$ 
then we see that by Theorem \ref{moment2} that (\ref{2nd}) is equal to 
$$\phi^*(q)T \hat{\psi}(0)L(2\sigma,\chi_{0,q})S_1((qT)^\kappa) + O_\epsilon\left( (qT)^{2-2\sigma}L(2-2\sigma,\chi_{0,q})|S_2((qT)^\kappa)| +(qT)^{1-\epsilon}\right).$$

To deal with $S_1$, we will need the following lemma. 
\begin{lemma}
\label{mob}
For all $t \geq 1$
$$\sum_{\substack{a \leq t\\ (a,n)=1}}\frac{\mu(a)}{a} \ll\frac{n}{\phi(n)}.$$
\end{lemma}
\begin{proof}
Define 
$$f_n(t): = \left|\sum_{\substack{a \leq t\\ (a,n)=1}}\frac{\mu(a)}{a}\right|$$
 and 
$$M_n:= \max_t \left|\sum_{\substack{a \leq t\\ (a,n)=1}}\frac{\mu(a)}{a}\right|.$$
 Then for any prime $p$ with $p^k||n$
$$f_n(t) = \left|\sum_{\substack{a \leq t \\ (a,n/p^k)=1}}\frac{\mu(a)}{a} - \frac{\mu(p)}{p}\sum_{\substack{ a \leq t/p \\ (a,n)=1}}\frac{\mu(a)}{a}\right| \leq M_{n/p^k} + \frac{1}{p}f_n(t/p).$$
Then by recursion, and as $\lim_{h\rightarrow \infty}f_n(t/p^h) = 0$ we see that 
$$f_n(t) \leq M_{n/p^k}\left(1 + p^{-1}+p^{-2}+...\right) = \frac{M_{n/p^k}}{1-p^{-1}}.$$
Then, by repeating this line of reasoning with each prime $p$ dividing $n$, we see that 
$$f_n(t) \leq M_1\frac{n}{\phi(n)}$$
but by the prime number theorem, $\sum_a\frac{\mu(a)}{a} = 0$, hence $M_1 \ll 1$. 
\end{proof}
We can now handle $S_1$ and $S_2$.

\begin{lemma}
\label{s1}
$$S_1(x) = L^{-1}(2\sigma, \chi_q)\left(1 + O\left((2\sigma-1)x^{1-2\sigma}\log^2(x)\right)\right)$$
\end{lemma}

\begin{proof}
We may assume that $\sigma$ is close to $1/2$ meaning that
$$L(2\sigma, \chi_{0,qn}) \asymp \frac{\phi_{2\sigma}(qn)}{(qn)^{2\sigma}(2\sigma - 1)}$$
where 
$$\phi_{2\sigma}(n) = \sum_{cd=n}\mu(c)d^{2\sigma} = n^{2\sigma}\prod_{p|n}\left(1-p^{-2\sigma}\right).$$

\begin{eqnarray*}
S_1(x) \kern-6pt& = &\kern-6pt \sum_{\substack{a,b,d\\(a,b)=1\\(abd,q)=1}}\frac{v(ad)v(bd)}{(abd)^{2\sigma}} =  \sum_{\substack{a,b,c,d\\ acd,bcd \leq x\\(abcd,q)=1}}\mu(c)\frac{v(acd)v(bcd)}{(abdc^2)^{2\sigma}}\\
 \kern-6pt& = &\kern-6pt \sum_{\substack{a,b\\(ab,q)=1}}\sum_{\substack{n\leq x\\ (n,q)=1}}\frac{v(an)v(bn)}{(abn^2)^{2\sigma}}\sum_{cd=n}\mu(c)d^{2\sigma}.
\end{eqnarray*}
Let

$$y_n = \sum_{\substack{a\\ (a,q) = 1}}\frac{v(an)}{(an)^{2\sigma}}$$
then 
\begin{equation*}
S_1(x) = \sum_{\substack{n \leq x \\(n,q) = 1}}y_n^2\phi_{2\sigma}(n).
\end{equation*}
Inserting the definition of $v(n)$ in to the definition of $y_n$ gives
\begin{eqnarray*}
y_n \kern-6pt& = &\kern-6pt \frac{\mu(n)}{n^{2\sigma}}\sum_{\substack{a \leq x/n\\(a,qn)=1}}\frac{\mu(a)}{a^{2\sigma}}\left(1-\left(\frac{x}{an}\right)^{1-2\sigma}\right)\left(1 - x^{1-2\sigma}\right)^{-1}\\
\kern-6pt& = &\kern-6pt  \frac{\mu(n)}{n^{2\sigma}}(2\sigma -1)\left(\frac{x}{n}\right)^{1-2\sigma}\left(1 - x^{1-2\sigma}\right)^{-1} \int_1^{x/n}\left(\sum_{\substack{a\leq t\\(a,qn)=1}}\frac{\mu(a)}{a^{2\sigma}}\right)t^{2\sigma - 2}dt
\end{eqnarray*}
by partial summation. As $2\sigma > 1$, the sum converges so we may write
$$\sum_{\substack{a\leq t\\(a,qn)=1}}\frac{\mu(a)}{a^{2\sigma}} = L^{-1}(2\sigma, \chi_{0,qn}) - \sum_{\substack{a> t\\(a,qn)=1}}\frac{\mu(a)}{a^{2\sigma}}.$$
As $2\sigma$ is close to 1, it is not sufficient to bound the error by $O\left(\frac{t^{1-2\sigma}}{2\sigma - 1}\right)$. Instead, we write
\begin{equation*}
\sum_{\substack{a> t\\(a,qn)=1}}\frac{\mu(a)}{a^{2\sigma}} = - \sum_{\substack{a\leq t\\(a,qn)=1}}\frac{\mu(a)}{a}t^{1-2\sigma} + (2\sigma - 1)\int_t^\infty \sum_{\substack{a\leq s\\(a,qn)=1}}\frac{\mu(a)}{a}s^{-2\sigma}ds.
\end{equation*}

So by Lemma \ref{mob}, for $q$ and $n$ co-prime 
$$\sum_{\substack{a> t\\(a,qn)=1}}\frac{\mu(a)}{a^{2\sigma}} \ll \frac{qn}{\phi(qn)}t^{1-2\sigma}.$$
This means 
$$\sum_{\substack{a\leq t\\(a,qn)=1}}\frac{\mu(a)}{a^{2\sigma}} = L^{-1}(2\sigma, \chi_{0,qn}) + O\left(\frac{qn}{\phi(qn)}t^{1-2\sigma}\right)$$
so 
\begin{eqnarray*}
y_n \kern-6pt& = &\kern-6pt \frac{\mu(n)}{n^{2\sigma}}(2\sigma -1)\frac{\left(\frac{x}{n}\right)^{1-2\sigma}}{1 - x^{1-2\sigma}}\int_1^{x/n}\left( L^{-1}(2\sigma, \chi_{0,qn}) + O\left(\frac{qn}{\phi(qn)}t^{1-2\sigma}\right)\right)t^{2\sigma - 2}dt\\
\kern-6pt& = &\kern-6pt \frac{\mu(n)}{n^{2\sigma}}\left(L^{-1}(2\sigma, \chi_{0,qn})\frac{1-\left(\frac{x}{n}\right)^{1-2\sigma}}{1 - x^{1-2\sigma}} + O\left(\left(\frac{x}{n}\right)^{1-2\sigma}\frac{qn(2\sigma-1)}{\phi(qn)}\int_1^{x/n}t^{-1}dt\right)\right)\\
\kern-6pt& = &\kern-6pt \frac{\mu(n)}{n^{2\sigma}(n)}L^{-1}(2\sigma, \chi_{0,qn})\left(\frac{1-\left(\frac{x}{n}\right)^{1-2\sigma}}{1 - x^{1-2\sigma}} + O\left(\left(\frac{x}{n}\right)^{1-2\sigma}\log(x/n)\right)\right)\\
\kern-6pt& = &\kern-6pt \frac{\mu(n)}{\phi_{2\sigma}(n)}L^{-1}(2\sigma, \chi_{q})\left(\frac{1-\left(\frac{x}{n}\right)^{1-2\sigma}}{1-x^{1-2\sigma}} + O\left(\left(\frac{x}{n}\right)^{1-2\sigma}\log(x/n)\right)\right)
\end{eqnarray*}
so 
$$y_n^2\phi_{2\sigma}(n) = \frac{\mu(n)^2}{\phi_{2\sigma}(n)}L^{-2}(2\sigma,\chi_q)\left(\left(\frac{1-\left(\frac{x}{n}\right)^{1-2\sigma}}{1-x^{1-2\sigma}}\right)^2 + O\left(\left(\frac{x}{n}\right)^{1-2\sigma}\log(x/n)\right)\right).$$
Note that for square-free $n$,
$$\frac{\mu(n)^2}{\phi_{2\sigma}(n)} = n^{-2\sigma}\prod_{p|n}\left(1 - p^{-2\sigma}\right)^{-1} = n^{-2\sigma}\prod_{p|n}\left(1 + p^{-2\sigma} + p^{-4\sigma} + p^{-6\sigma} + ...\right) = \sum_{\substack{m \\ \text{rad}(m) = n}}\frac{1}{m^{2\sigma}}$$
where rad($m$) $=\prod_{p|m}p$.Therefore
\begin{eqnarray*}
\sum_{\substack{n \leq x \\ (n,q)=1}} \frac{\mu(n)^2}{\phi_{2\sigma}(n)} \kern-6pt& =&\kern-6pt \sum_{\substack{m\\(m,q)=1\\ \text{rad}(m) < x}}\frac{1}{m^{2\sigma}}\\
\kern-6pt& =&\kern-6pt  \sum_{\substack{m \leq x\\ (m,q) = 1} }\frac{1}{m^{2\sigma}} + O\left(\sum_{\substack{m > x\\ (m,q) = 1} }\frac{1}{m^{2\sigma}}\right)\\
\kern-6pt& =&\kern-6pt L(2\sigma, \chi_q) + O\left(\frac{x^{1-2\sigma}\phi_{2\sigma}(q)}{(2\sigma - 1)q^{2\sigma}}\right)\\
\kern-6pt& =&\kern-6pt L(2\sigma, \chi_q) + O\left(x^{1-2\sigma}L(2\sigma, \chi_q)\right)
\end{eqnarray*}
So, supposing that $f(t)$ is a differentiable function with $f(x) = 0$, partial summation shows that 
$$\sum_{\substack{n \leq x \\ (n,q)=1}} \frac{\mu(n)^2}{\phi_{2\sigma}(n)}f(n) = L(2\sigma, \chi_q)\left(f(1) + O\left(\int_1^xt^{1-2\sigma}f'(t)dt\right)\right)$$
hence

$$\sum_{\substack{n \leq x \\ (n,q)=1}} \frac{\mu(n)^2}{\phi_{2\sigma}(n)}\left(\frac{1-\left(\frac{x}{n}\right)^{1-2\sigma}}{1-x^{1-2\sigma}}\right)^2 = L(2\sigma, \chi_q)$$
$$+ O\left(L(2\sigma, \chi_q)\int_1^x t^{1-2\sigma}2(2\sigma - 1)x^{1-2\sigma}t^{2\sigma - 2}\left(1 - \left(\frac{x}{t}\right)^{1-2\sigma}\right)dt\right)$$
$$ = L(2\sigma, \chi_q)\left(1+ O\left((2\sigma - 1)x^{1-2\sigma}\log(x)\right)\right)$$
and
$$\sum_{\substack{n \leq x \\ (n,q)=1}} \frac{\mu(n)^2}{\phi_{2\sigma}(n)}\left(\frac{x}{n}\right)^{1-2\sigma}\log(x/n)$$
 is
 \begin{eqnarray*}
 \kern-6pt &=&\kern-6pt L(2\sigma, \chi_q)x^{1-2\sigma}\log(x) + O\left(L(2\sigma, \chi_q)x^{1-2\sigma}\int_1^xt^{1-2\sigma}((2\sigma - 1)\log(x/t)-1)t^{2\sigma - 2}dt\right)\\
\kern-6pt &=&\kern-6pt L(2\sigma, \chi_q)\left(x^{1-2\sigma}\log(x) + O\left((2\sigma - 1)x^{1-2\sigma}\log^2(x)\right)\right).
\end{eqnarray*}
Hence, 
$$S_1(x) = \sum_{\substack{n \leq x\\ (n,q) = 1}}y_n^2\phi_{2\sigma}(n) = L^{-1}(2\sigma, \chi_q)\left(1 + O\left((2\sigma - 1)x^{1-2\sigma}\log^2(x)\right)\right)$$
\end{proof}
We now bound $S_2(x)$.
\begin{lemma}
\label{s2}
 $$S_2(x) \ll L(2\sigma, \chi_q)\log(x)^2 $$
\end{lemma}
\begin{proof}

Similar to before, we see that 
$$S_2(x) = \sum_{\substack{n \leq x\\ (n,q) = 1}}y_n^2 \phi_{2-2\sigma}(n)$$
with 
\begin{eqnarray*}
y_n\kern-6pt& :=&\kern-6pt \frac{\mu(n)}{n}\sum_{\substack{a \leq x/n\\ (a,qn)=1}}\frac{\mu(a)}{a}\left(1 - \left(\frac{x}{an}\right)^{1-2\sigma}\right)\left(1-x^{1-2\sigma}\right)^{-1}\\
\kern-6pt& \ll &\kern-6pt \frac{1}{n}\sum_{a \leq x/n}\frac{1}{a}\\
\kern-6pt& \ll &\kern-6pt \frac{\log(x)}{n}
\end{eqnarray*}
so 
\begin{eqnarray*}
\sum_{\substack{n \leq x\\(n,q)=1}}y_n^2\phi_{2-2\sigma}(n) \kern-6pt& \ll & \kern-6pt \sum_{\substack{n \leq  x\\ (n,q)=1}} \frac{\phi_{2-2\sigma}(n)}{n^2}\log(x)^2\\
\kern-6pt& \ll & \kern-6pt \sum_{\substack{n \leq x\\(n,q)=1}}\frac{\log(x)^2}{n^{2\sigma}}\\
\kern-6pt& \ll & \kern-6pt L(2\sigma, \chi_q)\log(x)^2
\end{eqnarray*}
\end{proof}

Hence by lemmas \ref{s1} and \ref{s2}, we arrive at the conclusion that \ref{2nd} is equal to 
\begin{equation}
    \label{2ndmoment}
\phi^*(q)T\hat{\psi}(0) + O_\epsilon\left((qT)^{2-2\sigma}\log^4(qT) + (2\sigma -1 )(qT)^{1+\kappa(1-2\sigma)}\log(qT)^2 + (qT)^{1-\epsilon}\right).
\end{equation}

We turn our attention to the first moment.
\begin{lemma}
\begin{equation}
\label{eqn: 1st}
\int\sideset{}{^*}\sum_{\chi\Mod{q}}L(\sigma+it,\chi)M(\sigma+it,\chi)\psi(t/T)dt = \phi^*(q)T\hat{\psi}(0) + O\left((qT)^{\epsilon + \kappa(1-\sigma)}\right).
\end{equation}
\end{lemma}

\begin{proof}
Suppose that $\chi$ is a primitive character of conductor $q$, $\sigma \in [1/2, 1]$, $t \in [T,2T]$, then

Define 
$$A:= \sum_{n}\frac{\chi(n)}{n^{\sigma + it}}e^{-\frac{n}{(qT)^2}}$$
Then, as 
$$e^{-x} = \frac{1}{2\pi i}\int_{(1)}\Gamma(s)x^{-s}ds$$
$A$ may be written as 
$$\frac{1}{2\pi i}\int_{(1)}(qT)^{2s}\Gamma(s)L(\sigma + it + s,\chi)ds.$$

Moving the contour of integration to have real part $-1 + \epsilon$, we hit a pole at $s = 0$. The integral at the new contour may be bounded by the exponential decay of the Gamma function, and by the functional equation for Dirichlet $L$-functions, 
$$L(\sigma - 1 + \epsilon + it, \chi) \ll (qT)^{\frac{3}{2}-\sigma - \epsilon}|L(2-\sigma - \epsilon - it, \bar{\chi})| \ll (qT)^{1-\epsilon}.$$
Hence
\begin{equation}
\label{eqn: L}
L(\sigma + it,\chi) = A + O\left((qT)^{-1+\epsilon}\right).
\end{equation}

By (\ref{eqn: L}), $\int\sideset{}{^*}\sum_{\chi\Mod{q}}L(\sigma+it,\chi)M(\sigma+it,\chi)\psi\left(\frac{t}{T}\right)dt$ is equal to 
$$\sum_{w|q}\mu\left(\frac{q}{w}\right)\phi(w)\sum_{\substack{n \leq (qT)^{2+\epsilon}\\ a\leq (qT)^\kappa\\ an \equiv 1 \ (w)\\(an,q)=1)}}\frac{v(a)e^{-\frac{n}{(qT)^2}}}{(an)^{\sigma}}\int(an)^{-it}\psi\left(\frac{t}{T}\right)dt +O_\epsilon\left((qT)^{\epsilon + \kappa(1-\sigma)}\right)$$
If $an > 1$ then by integration by parts $K$ times
$$\int(an)^{it}\psi(t/T)dt \ll_K \log(an)^{-K}T^{1+\epsilon -K}$$
so we may make the error term arbitrarily small. When $an=1$ then the integral is just $T\hat{\psi}(0)$. Hence
\begin{equation*}
\int\sideset{}{^*}\sum_{\chi\Mod{q}}L(\sigma+it,\chi)M(\sigma+it,\chi)\psi(t/T)dt = \phi^*(q)T\hat{\psi}(0) + O\left((qT)^{\epsilon + \kappa(1-\sigma)}\right).
\end{equation*}
\end{proof}
By (\ref{2ndmoment}) and (\ref{eqn: 1st}), we see that 
$$\int\sideset{}{^*}\sum_{\chi\Mod{q}}|L(\sigma+it,\chi)M(\sigma+it,\chi)-1|^2\psi(t/T)dt$$
is bounded by
$$O_\epsilon\left((qT)^{2-2\sigma}\log^4(qT) + (2\sigma -1 )(qT)^{1+\kappa(1-2\sigma)}\log(qT)^2 + (qT)^{1-\epsilon}\right)$$

which concludes the proof of Proposition \ref{density} and Theorem \ref{theoremdensity}.

\section{Proof of Theorem \ref{Levinson}}
\label{Levinson proof}
Levinson's original proof was long and allegedly had a reputation for being difficult. In this section we shall follow the elegant reformulation of the method by Young in \cite{LevinsonY}, but in the context of families of Dirichlet $L$-functions. Assume the conditions of Theorem \ref{Levinson} and let $L= \log(qT)$, and 
$$V_\chi(s) = Q\left(-\frac{1}{L}\frac{d}{ds}\right)L(s,\chi).$$
Suppose that $M(s,\chi)$ is a mollifier of the form 
$$M(s,\chi) = \sum_{a \leq X}\frac{\chi(a)\mu(a)}{a^{s}}P\left(\frac{\log(X/a)}{\log(X)}\right)$$
where $P(x) = \sum_ia_ix^i$ with $P(0)=0, P(1)=1$ and for convenience we shall write $P[a] = P\left(\frac{\log(X/a)}{\log(X)}\right)$. Levinson's method (see for example Corollary A of \cite{CIS}) shows that 
\begin{equation}
\label{eqn: main}
\frac{N_0(T,q)}{N(T,q)} \geq 1 - \frac{1}{R}\log \left(\frac{1}{\phi^*(q)T}\int_1^T\sideset{}{^*}\sum_{\chi\Mod{q}}\left|V_\chi\left(\frac{1}{2}-\frac{R}{L} + it\right)M\left(\frac{1}{2} + it\right)\right|^2dt\right) + o(1)
\end{equation}
as $qT \rightarrow \infty$. Additionally, restricting $Q(x)$ to be a linear polynomial restricts $N_0(T,q)$ to only counting simple zeros. Defining 
$$I(\alpha,\beta) = \sideset{}{^*}\sum_{\chi\Mod{q}}\int_{\mathbb{R}}L(1/2 + \alpha + it,\chi)L(1/2 + \beta - it,\chi)|M(1/2 + it,\chi)|^2\psi(t/T)dt$$
for $\alpha,\beta \ll L^{-1}$, and for a smooth function $\psi(t)$ supported on $[T,2T]$ we arrive at the integral in (\ref{eqn: main}) by evaluating
$$Q\left(-\frac{1}{L}\frac{d}{d\alpha}\right)Q\left(-\frac{1}{L}\frac{d}{d\beta}\right)I(\alpha,\beta)$$
at $\alpha = \beta = -R/L$. Applying Theorem (\ref{moment2}) for $X= (qT)^\kappa$ with $\kappa < 69/128$ we see that 
$$\frac{I(\alpha, \beta)}{\phi^*(q)T} = \hat{\psi}(0) L(1 + \alpha + \beta, \chi_{0,q})\sum_{\substack{da,db \leq (qT)^\kappa\\ (a,b)=1\\ (abd,q) = 1}} \frac{\mu(ad)P[ad]\mu(bd)P[bd]}{a^{1+\beta}b^{1+\alpha}d}$$
$$ + \hat{\psi}(0)\left(\frac{qT}{\pi}\right)^{-\alpha-\beta}\left(1+O(L^{-1})\right)L(1-\alpha - \beta, \chi_{0,q})\sum_{\substack{da,db \leq (qT)^\kappa\\ (a,b)=1\\(abd,q)=1}} \frac{\mu(ad)P[ad]\mu(bd)P[bd]}{a^{1-\alpha}b^{1-\beta}d} $$
$$+O_\delta\left((qT)^{-\delta}\right)$$
for a positive constant $\delta > 0$, as the ratio of gamma functions in the integral is $t^{-\alpha - \beta}(1+O(t^{-1})) = T^{-\alpha-\beta}(1+O(\log(T)^{-1}))$ for $t\in [T,2T]$ by Lemma \ref{gamma}, and by the assumption that $T \gg q^\epsilon$ we have $\log(T)\gg \log(q)$.  
Let
$$S(\alpha,\beta) =  L(1 + \alpha + \beta, \chi_{0,q})\sum_{\substack{da,db \leq (qT)^\kappa\\ (a,b)=1\\ (abd,q) = 1}} \frac{\mu(ad)P[ad]\mu(bd)P[bd]}{a^{1+\beta}b^{1+\alpha}d}$$
so 
\begin{eqnarray*}
\frac{I(\alpha,\beta)}{\phi^*(q)T} \kern-6pt&=& \kern-6pt \hat{\psi}(0)\left(S(\alpha,\beta) + \left(qT\right)^{-\alpha-\beta}S(-\beta,-\alpha)(1+O(L^{-1}))\right) + O_\delta\left((qT)^{-\delta}\right)\\
\kern-6pt&=& \kern-6pt \hat{\psi}(0)\left(S(\alpha,\beta) + \left(qT\right)^{-\alpha-\beta}S(-\beta,-\alpha)\right) +O(L^{-1})
\end{eqnarray*}
as long as $S(-\beta,-\alpha) \ll 1$ which we shall show is the case in the following lemma.
\begin{lemma}
Uniformly on any fixed annuli such that $\alpha,\beta \asymp L^{-1}, |\alpha+\beta| \gg L^{-1}$
$$S(\alpha,\beta) = \frac{1}{(\alpha + \beta)\log(X)}\frac{d^2}{dxdy}X^{\alpha x + \beta y}\int_0^1 P(x+u)P(y+u)du|_{x=y=0}+ O(L^{-1}) $$
\end{lemma}

\begin{proof}
For $1 \leq a \leq X$ and $i \in \mathbb{N}$ 
\begin{equation*}
\frac{i!}{\log(X)^i}\frac{1}{2\pi i}\int_{(1)}\left(\frac{X}{a}\right)^v\frac{dv}{v^{i+1}} = \begin{cases}
\left(\frac{\log(X/a)}{\log(X)}\right)^i & \text{if }  1\leq a \leq X \\
0 & \text{if }  a > X
\end{cases}
\end{equation*}
hence $S(\alpha,\beta)$ is  equal to 
$$L(1+\alpha+\beta, \chi_{0,q})\sum_{i,j}\frac{a_ia_ji!j!}{\log(X)^{i+j}}\frac{1}{(2\pi i)^2}\int_{(1)}\int_{(1)}X^{u+v}\sum_{\substack{a,b,d\\(a,b)=1\\ (abd,q)=1}}\frac{\mu(ad)\mu(bd)}{a^{1+\beta + u}b^{1+\alpha + v}d^{1+u+v}}\frac{du}{u^{i+1}}\frac{dv}{v^{j+1}}.$$
By considering Euler products,
\begin{equation}
\label{eqn: A}
L(1+\alpha+\beta,\chi_{0,q})\sum_{\substack{a,b,d\\ (a,b)=1\\(abd,q)=1}}\frac{\mu(ad)\mu(bd)}{a^{1+\beta + u}b^{1+\alpha + v}d^{1+u+v}} = \frac{\zeta(1+\alpha+\beta)\zeta(1+u+v)A_{\alpha,\beta}(u,v)}{\zeta(1+\alpha +v)\zeta(1+\beta+u)}
\end{equation}
where $A_{\alpha,\beta}(u,v)$ is an absolutely convergent Euler product in some product of half planes containing the origin. If we can show that $A_{0,0}(0,0)=1$, then we may appeal to Lemma 7 from \cite{LevinsonY} to show that
$$\frac{1}{(2\pi i)^2}\int_{(1)}\int_{(1)}X^{u+v}\frac{\zeta(1+u+v)A_{\alpha,\beta}(u,v)}{\zeta(1+\alpha +v)\zeta(1+\beta+u)}\frac{du}{u^{i+1}}\frac{dv}{v^{j+1}} $$
$$= \frac{(\log(X))^{i+j-1}}{i!j!}\frac{d^2}{dxdy}X^{\alpha x + \beta y}\int_0^1(x+u)^i(y+u)^jdu|_{x=y=0}+O(L^{i+j-1})$$
at which point we may sum over $i$ and $j$, and take a Taylor expansion of $\zeta(1+\alpha+\beta)$ to obtain the desired result.

All that remains to be shown is that $A_{0,0}(0,0) = 1$. Suppose that $\alpha = \beta = u = v = s>0$, then by (\ref{eqn: A})
\begin{eqnarray*}
A_{s,s}(s,s)\kern-6pt& = &\kern-6pt L(1+2s,\chi_{0,q})\sum_{\substack{a,b,d\\(a,b)=1\\ (abd,q)=1}}\frac{\mu(ad)\mu(bd)}{(abd)^{1+2s}}\\
\kern-6pt& = &\kern-6pt  \sum_{\substack{a,b,d,n\\(a,b)=1\\ (abdn,q)=1}}\frac{\mu(ad)\mu(bd)}{(abdn)^{1+2s}}
\end{eqnarray*}
by the Dirchlet series of the $L$-function. Re-labelling $a = ad, \ b = bd, \ m = bn, \ n = an$, we may write the sum as
$$\sum_{\substack{am=bn \\(abmn,q)=1}}\frac{\mu(a)\mu(b)}{(abmn)^{1/2 +s}} = 1$$
by the M{\"o}bius formula.

Hence $A_{s,s}(s,s) = 1$ for all $s>1$. As the Euler product converges absolutely at the origin, $A_{0,0}(0,0) = \lim_{s\rightarrow 0}A_{s,s}(s,s) = 1$. 
\end{proof}

We have now arrived at the equivalent of Lemma 6 in \cite{LevinsonY}. By precisely the same method as Young's we may arrive at the following proposition.
\begin{proposition}
$$\frac{1}{\phi^*(q)T}\int_1^T\sideset{}{^*}\sum_{\chi\Mod{q}}\left|V_\chi(1/2-R/L + it)M(1/2 + it)\right|^2dt = c(P,Q,R) + O(L^{-1})$$
where 
$$c(P,Q,R) = 1+ \frac{1}{\kappa}\int_0^1\int_0^1 e^{2Rv}\left(\frac{d}{dx}e^{R\kappa x}P(x+u)Q(v+\kappa x)|_{x=0}\right)^2dudv$$
for some positive constant $R$.
\end{proposition}

The next step is to choose $R,P,$ and $Q$ to maximise 
$$ 1-\frac{1}{R}\log(c(P,Q,R))$$
subject to the conditions that $R$ is a positive constant, $P(0)=0, P(1)=1$ and $Q(0)=1$. We shall stipulate that $Q$ is a linear polynomial, in order to determine a lower bound on the proportion of simple zeros on the critical line. The optimisation process can be found in Section 4 of Conrey's paper \cite{Conrey1989}. This method demonstrates that the optimal choice for $P(x)$ is of the form
$$P(x) = \frac{e^{rx}-e^{sx}}{e^r-e^s}$$
for $r,s$ constants. While this is not a polynomial, it may be uniformly approximated by real polynomials.  Choosing 
$$Q(x) = 1 - 1.035x, \ R= 1.179$$
gives 
$$1-\frac{1}{R}\log(c(P,Q,R)) = 0.382156$$
and hence 
$$\frac{N_0(T,q)}{N(T,q)} \geq 0.382$$
for large enough $qT$.

\bibliographystyle{acm}
\bibliography{biblio}

\end{document}